\documentclass[a4paper,12pt]{amsart}

\usepackage{amssymb}
\usepackage{amsmath}
\usepackage{amsthm}

\usepackage{mathrsfs}

\allowdisplaybreaks
\numberwithin{equation}{section}

\theoremstyle{plain}
\newtheorem{theorem}{Theorem}[section]
\newtheorem{prop}[theorem]{Proposition}
\newtheorem{corollary}[theorem]{Corollary}
\newtheorem{lemma}[theorem]{Lemma}

\theoremstyle{definition}
\newtheorem{defn}[theorem]{Definition}
\newtheorem{rem}[theorem]{Remark}

\newcommand{\be}{\begin{equation}}
\newcommand{\ee}{\end{equation}}
\def\R{{\mathbb R}}
\def\N{{\mathbb N}}
\def\Z{{\mathbb Z}}
\def\C{{\mathbb C}}
\def\Rn{{{\mathbb R}^n}}
\def\Tn{{{\mathbb T}^n}}
\def\Zn{{{\mathbb Z}^n}}

\def\supp{\;{\rm supp}\;}
\def\dslash{\ {\rm d}\llap {\raisebox{.9ex}{$\scriptstyle-\!$}}}
\def\irm{{\rm i}}
\def\erm{{\ {\rm e}}}
\def\drm{{\ {\rm d}}}
\def\Scal{{\mathcal S}}

\def\Dcal{{\mathcal D}}
\def\Lcal{{\mathcal L}}
\def\Pcal{{\mathcal P}}

\def\Hcal{{\mathcal H}}
\def\Gcal{{\mathcal G}}
\def\FT{{\mathscr F}}
\def\FTT{{\FT_\Tn}}
\def\FTR{{\FT_\Rn}}
\def\va{\varphi}
\def\Lap{{\mathscr L}}

\def\p#1{{\left({#1}\right)}}
\def\b#1{{\left\{{#1}\right\}}}
\def\br#1{{\left[{#1}\right]}}
\def\jp#1{{\left\langle{#1}\right\rangle}}
\def\n#1{{\left\|{#1}\right\|}}
\def\abs#1{{\left|{#1}\right|}}

\begin{document}

\title{Quantization of pseudo-differential operators
on the torus}

\author[Michael Ruzhansky]{Michael Ruzhansky}
\address{
  Michael Ruzhansky:
  \endgraf
  Department of Mathematics
  \endgraf
  Imperial College London
  \endgraf
  180 Queen's Gate, London SW7 2AZ 
  \endgraf
  United Kingdom
  \endgraf
  {\it E-mail address} {\rm m.ruzhansky@imperial.ac.uk}
  }
  
\author[Ville Turunen]{Ville Turunen}
\address{
  Ville Turunen:
  \endgraf
   Helsinki University of Technology
  \endgraf
  Institute of Mathematics
  \endgraf
   P.O. Box 1100
  \endgraf
   FI-02015 HUT
  \endgraf
  Finland
  \endgraf
  {\it E-mail address} {\rm ville.turunen@hut.fi}
  }

\thanks{The first
 author was supported in part by the JSPS Invitational
 Research Fellowship.}
\date{\today}


\begin{abstract}
Pseudo-differential and Fourier series operators on the torus
$\Tn=(\Bbb R/2\pi\Bbb Z)^n$ are analyzed by using
global representations by Fourier series
instead of local representations in coordinate charts.
Toroidal symbols are investigated and 
the correspondence between toroidal and Euclidean symbols
of pseudo-differential operators is established. Periodization
of operators and hyperbolic partial differential equations
is discussed. Fourier series operators, which are analogues
of Fourier integral operators on the torus, are introduced,
and formulae for their compositions with pseudo-differential 
operators are derived. It is shown that
pseudo-differential and 
Fourier series operators are bounded on $L^2$ under
certain conditions on their phases and amplitudes.
\end{abstract}

\maketitle

\tableofcontents

\section{Introduction}

In this paper we investigate a global quantization of 
pseudo-differential operators on the torus. For this, 
we develop necessary elements of the 
toroidal microlocal analysis.
Using the toroidal Fourier transform (Fourier series) we 
can quantize the usual classes of pseudo-differential operators
yielding a full symbol of an operator on the torus.
The main difference with the (Kohn--Nirenberg)
quantization in H\"ormander's
symbol classes $S^m_{\rho,\delta}(\Tn\times\Rn)$ 
(which we will often call {\em Euclidean quantization} to
emphasize the difference) is that
the toroidal symbol belongs to the corresponding
symbol class $S^m_{\rho,\delta}(\Tn\times\Zn)$ with
frequency variable $\xi\in\Zn$ in the integer lattice;
this quantization will be called the {\em toroidal
quantization}. We will
analyze the relation between these symbol classes and between
corresponding pseudo-differential operators.
  
We will also investigate
the corresponding toroidal version of Fourier
integral operators. To distinguish them from those defined
using the Euclidean Fourier transform, we will call them
{\it Fourier series operators}. The use of 
Fourier series will allow us to obtain a global representation
of these operators, thus making trivial a number of topological
obstructions known in the standard theory of
Fourier integral operators on manifolds. We will
prepare the machinery and describe how it can be further
used in the calculus of Fourier series operators and
applications to hyperbolic partial differential equations.
In fact, the form of the required discrete calculus 
is not a-priori clear. For example, composition
formulae for pseudo-differential operators rely on the
Taylor expansion of symbols which are now defined on the
discrete lattice $\Zn$. Thus, 
we develop the corresponding versions of the periodic
and discrete analysis which are still quite
similar to formulations in the standard Euclidean theory.
In particular, this will include the analysis of
differences and some elements of the microlocal analysis
such as toroidal wave front sets, etc.
We study pseudo-differential operators in detail giving
an explicit relation between Euclidean and toroidal symbols.
Moreover, we will describe how this relation can be used
to relate Euclidean and toroidal quantizations of the
the same operators using a periodization operator that
we analyze for this purpose.

However, our analysis will reveal certain limitations for the
development of an equivalent full 
theory for Fourier series operators.
The main difference with pseudo-differential operators is in the
behaviour of the wave front set under the action of
operators. In the case of an pseudo-differential operator,
the wave front set
does not move and so we remain in the space-frequency space
$\Tn\times\Zn$. In the case of a Fourier series operator,
their integral kernel may have singularities away from the
diagonal in $\Tn\times\Tn$, which means that under the 
flow the wave front set may no longer be a transformation of
$\Tn\times\Zn$. Nevertheless, we will be able to resolve
this issue in the composition formulae by using (unique modulo
smoothing) extensions of toroidal symbols.

It was realized already in the 1970s that on the torus
one can study pseudo-differential operators globally
using Fourier series expansions,
in analogy to the Euclidean pseudo-differential calculus.
These {\it periodic pseudo-differential operators}
were treated e.g. by Agranovich \cite{Agranovich1,Agranovich2}.
Contributions have been made by many authors, e.g. by
Amosov \cite{Amosov}, Elschner \cite{Elschner}, 
McLean \cite{McLean}, Melo \cite{Me97},
Pr\"ossdorf and Schneider
\cite{ProssdorfSchneider}, Saranen and Wendland
\cite{SaranenWendland}, Turunen \cite{Turunen}, 
Turunen and Vainikko \cite{TurunenVainikko}, 
Vainikko and Lifanov \cite{VainikkoLifanov1,
VainikkoLifanov2}, and others.
However, most of these papers deal with one-dimensional cases
or with classes of operators rather than with classes
of symbols. In this paper we in particular develop the symbolic
analysis of pseudo-differential operators on the torus thus
closely relating Euclidean and toroidal quantizations
of the same classes of operators.

We note that the use of operators which are discrete in
the frequency variable allows one to weaken regularity 
assumptions on symbols with respect to $\xi$-variables.
Symbols with low regularity in $x$ have been under intensive
study for many years, e.g. see
Kumano-go and Nagase \cite{Kumanogonagase}, Sugimoto 
\cite{Sugimoto}, Boulkhemair \cite{Boulkhemair}, 
Garello and Morando \cite{GarelloMorando}, and many other
contributions.
However, in these papers one assumes symbols to be smooth
or sufficiently regular in $\xi$.
The discrete approach in this paper will allow us to reduce
or to completely remove the
regularity assumptions with respect to $\xi$. For example,
no regularity with respect to $\xi$ is assumed for $L^2$
estimates, and for elements of the calculus.
We note that a notion of pseudo-differential operators with
low regularity in $\xi$ can still be introduced, 
as was done e.g. by Sj\"ostrand \cite{Sj94}, or with 
symbols in more general modulation spaces, see e.g. 
Gr\"ochenig \cite{Gr01}. Although such
operators form an algebra, no explicit symbolic calculus of such
operators is available. In this respect the toroidal
quantization has an advantage of looking only at discrete
values of $\xi$, thus removing the regularity issue 
altogether, and yielding the symbolic calculus.

We note that the approach developed in this paper may be 
generalised to other (non-commutative) Lie groups. Thus, we will
develop global quantizations of pseudo-differential operators
on ${\mathbb S}^3$ in \cite{RT-S3} by identifying it with
group SU(2) via the quaternionic product in $\R^4$. Moreover,
in \cite{RT-book}
we will also develop global quantizations of
pseudo-differential operators on arbitrary compact
Lie groups and on certain homogeneous spaces without resorting to
local coordinate charts. We note that this analysis is rather
different from the one developed by Taylor 
in \cite{Taylor84} because of the obtained
descriptions of symbol classes and their transformations.
In any case, the relation between compact Lie groups and $\Rn$
is most transparent in the case of 
the torus and the utilisation of this
relation brings considerable simplifications in the analysis
of this paper compared to that in \cite{RT-S3} and
\cite{RT-book}, which rely heavily on the
representation theory. 
Some results proved in this paper were announced
in \cite{RT07}.

We fix the notation for the torus as
$\Tn=(\R/2\pi\Z)^n$.
Often we may identify $\Tn$ with the cube
$[0,2\pi)^n\subset\Rn$ (or $[-\pi,\pi)^n$),
where we identify the measure on the torus 
with the restriction of
the Euclidean measure on the cube.
Functions on $\Tn$ can be thought as those 
functions on $\Rn$
that are $2\pi$-periodic in each of the coordinate directions.
We will often simply say that such functions are $2\pi$-periodic
(instead of $2\pi\Zn$-periodic).

For $N\in\N$ we will
also use the notation $N\Tn=(\R/2\pi N\Z)^n$ and call
this the {\em $N$-inflated torus}, or simply an 
{\em inflated torus} if
the value of $N$ is not of importance.
One of the reasons to use large $N$ in this analysis
is that if one wants to embed a compactly supported problem in $\Rn$
into a torus, one may need to choose $N$ large enough for
technical simplifications. We will briefly discuss
these issues.

One particular application discussed in this paper will be to
hyperbolic equations on the torus. In particular, the developed
analysis can be applied if we embed certain problems in 
$\Rn$ into the torus. In general, we can observe that due to
the finite propagation speed of singularities in hyperbolic
equations we can usually cut-off the equation and the Cauchy data
for large $x$ in the local analysis of singularities of 
solutions for bounded times. Then we can embed the problem
into $\Tn$ using rescaling (or into the $N$-inflated torus
$N\Tn$ without rescaling) in order to apply the periodic
analysis developed in this paper. One of the advantages
of this procedure is that since phases and amplitudes now
are evaluated only at $\xi\in\Zn$ one can apply this also for
problems with low regularity in $\xi$. This type of
application is particularly important for weakly hyperbolic
equations or systems with variable multiplicities when 
at points of multiplicities characteristics become irregular. 
If the
principal part has constant coefficients then the loss of
regularity occurs only in $\xi$, so techniques developed in
this paper can be applied. Such applications to systems
will appear elsewhere and in this paper we 
shall briefly discuss only
the case of scalar equations. 

We also note that this
approach to the representation of solutions is related to
the work of Bourgain \cite{Bo1, Bo2, Bo3} on Strichartz
estimates for solutions to Schr\"odinger equations on the 
torus $\Tn$. Bourgain used the Fourier series representation
$$
  u(t,x)=\erm^{\irm 
  t\Delta} f(x)=\sum_{\xi\in\Zn} 
  \erm^{\irm(x\cdot\xi- t|\xi|^2)} 
  \widehat{f}(\xi),
$$
which shows in addition that the solution is periodic in time.
This representation allowed him to prove, for example, the equality
$\n{u}^4_{L^4(\Bbb T\times\Tn)}=\n{u^2}^2_{L^2(\Bbb T\times\Tn)}$
leading to the corresponding Strichartz estimates and global
well-posedness results for nonlinear equations (while
dispersive estimates fail even locally in time). In Section
\ref{SEC:torus-hyperbolic} we discuss solutions to  
hyperbolic equations with variable
coefficients but leaving corresponding
nonlinear applications outside the scope of this paper.

Overall, the analysis developed in this paper lays
down a foundation for tackling a variety of global problems
on the torus with the use of the full toroidal symbols (questions
like global hypoellipticity, global solvability,
estimates for pseudo-differential operators, Fefferman--Phong
inequality, etc.), and for applying these techniques to
partial differential equations. Again, we leave further
applications outside the scope of this paper.

In Section \ref{SEC:preliminaries} we fix the notation
used throughout the paper. In Section \ref{SEC:diff-Taylor}
we develop the necessary machinery for the discrete calculus,
in particular we prove the Taylor expansion formula
on the lattice $\Zn$ giving an estimate for the remainder.
Section \ref{SEC:torus-symbols} is devoted to the notion and
main properties of toroidal symbols together with formulae
for their calculus. In Section \ref{SEC:ext-symbols}  we establish
the relation between Euclidean and toroidal symbols and in
Section \ref{SEC:periodization} the relation between operators
with Euclidean and toroidal quantizations. 
In Section \ref{SEC:torus-WF} we prove basic results about
toroidal wave front sets. Section \ref{SEC:FSO-calculus}
introduces Fourier series operators and establishes composition
formulae between Fourier series and pseudo-differential operators.
In Section \ref{SEC:torus-L2} we prove results on the
$L^2$-boundedness of pseudo-differential and Fourier series
operators. Finally, Section \ref{SEC:torus-hyperbolic}
is devoted to an application to hyperbolic equations 
showing a way how to embed such equations into the
torus by periodization.

We will use the following notation in the paper.
Triangles $\triangle$ and $\overline\triangle$
will denote the forward and backward difference 
operators, respectively. 
The Laplacian will be denoted by $\Lap$.
The Dirac delta at $x$ will be denoted by $\delta_x(y)$ and
the Kronecker delta at $\xi$ will be denoted by $\delta_{\xi,\eta}$.
We will use the standard notation for
multi-indices $\alpha\in\N_0^n$, with
$\N_0 = \{0\}\cup\N$, for example
denoting $|\alpha| := \sum_{j=1}^n |\alpha_j|$,
$\alpha!=\alpha_1 !\cdots\alpha_n !$, 
$\partial_x^\alpha=\partial_{x_1}^{\alpha_1}\cdots
\partial_{x_n}^{\alpha_n}$, etc. We will also use
the standard notation $D_{x_j}=-\irm\partial_{x_j}=
\frac{1}{\irm}\frac{\partial}{\partial x_j}$,
where $\irm=\sqrt{-1}$.
We will write $a\in C^\infty(\Tn\times\Zn)$ when function
$a(\cdot,\xi)$ is smooth on $\Tn$ for all $\xi\in\Zn$.
For symbol classes $S^m_{\rho,\delta}$, we will often write
simply $S^m$ for the class $S^m_{1,0}$. For $\xi\in\Rn$,
we will write $\jp{\xi}=(1+|\xi|^2)^{1/2}$.
To avoid keeping track of constants in integrals we will
normalise the measure by
defining $\dslash x = (2\pi)^{-n}\ {\rm d}x$.
Constants will be usually denoted by $C$ (sometimes with
subscripts), and their values may differ on different
occasions, even when appearing in subsequent estimates.

\section{Preliminaries}
\label{SEC:preliminaries}

In this section we will introduce some notation which 
will be frequently used in the sequel.
Let $\Scal(\Rn)$ denote the space of the Schwartz test functions,
and let $\Scal'(\Rn)$ be
its dual, i.e. the space of the tempered distributions.
The {\it Dirac delta comb}
$\delta_{\Zn}:{\Scal}(\Rn)\to\Bbb C$ is defined by
$$
  \langle \delta_{\Zn},\va \rangle := 
  \sum_{x\in\Zn} \va(x),
$$
and the sum here is absolutely convergent.

Let $\Scal(\Zn)$ denote the space of
{rapidly decaying functions $\Zn\to\Bbb C$}.
That is, $\va\in\Scal(\Zn)$ if for any $M<\infty$
there exists a constant $C_{\va,M}$ such that
$
  |\va(\xi)|\leq C_{\va,M} \langle\xi\rangle^{-M}
$
holds for all $\xi\in\Zn$.
The topology on $\Scal(\Zn)$
is given by the seminorms $p_k$,
where $k\in\Bbb N_0$ and
$
  p_k(\va):=\sup_{\xi\in\Zn} 
  \langle\xi\rangle^k \left|\va(\xi)\right|.
$
One can show that the continuous linear functionals
on $\Scal(\Zn)$ are of the form
$$
  \va\mapsto \langle u,\va\rangle := \sum_{\xi\in\Zn} u(\xi)\ \va(\xi),
$$
where functions $u:\Zn\to\Bbb C$
grow at most polynomially at infinity,
i.e. there exist constants $M<\infty$ and $C_{u,M}$
such that
$$
  |u(\xi)| \leq C_{u,M} \langle\xi\rangle^M
$$
holds for all $\xi\in\Zn$.
Such distributions $u:\Zn\to\Bbb C$ form the space
$\Scal'(\Zn)$.

Let $\FTR:\Scal(\Rn)\to\Scal(\Rn)$
be the Euclidean Fourier transform
defined by
$$
  \p{\FTR f}(\xi)
  := \int_{\Rn} {\rm e}^{-{\rm i}x\cdot\xi}\ f(x)
  \dslash x,
$$
where $\dslash x = (2\pi)^{-n}\ {\rm d}x$.
Mapping $\FTR:\Scal(\Rn)\to\Scal(\Rn)$ is a bijection,
and its inverse $\FT_\Rn^{-1}$ is given by 
$$
  f(x) = \int_{\Rn} 
   {\rm e}^{{\rm i}x\cdot\xi}\ \p{\FTR f}(\xi)
  \drm\xi.
$$
As it is well-known, this Fourier transform can be uniquely extended to
$\FTR: \Scal'(\Rn)\to\Scal'(\Rn)$ by duality.

Let $\FTT: C^\infty(\Tn)\to\Scal(\Zn)$
be the {toroidal Fourier transform}
defined by
\begin{equation}\label{EQ:torus-FTT}
   (\FTT f)(\xi) 
  := \int_{\Tn} {\rm e}^{-{\rm i}x\cdot\xi}\ f(x)
   \dslash x,
\end{equation}
where $\dslash x = (2\pi)^{-n}\ {\rm d}x$.
Then $\FTT$ is a bijection and its inverse 
$\FT_\Tn^{-1}:\Scal(\Zn)\to C^\infty(\Tn)$
is given by
$$
 \p{\FT_\Tn^{-1}g}(x)=\sum_{\xi\in\Zn}  
 {\rm e}^{{\rm i}x\cdot\xi}\ g(\xi),
$$
so that
$
  f(x) = \sum_{\xi\in\Zn} 
   {\rm e}^{{\rm i}x\cdot\xi}\ \p{\FTT f}(\xi).
$
By dualising $\FT^{-1}_\Tn:\Scal(\Zn)\to C^\infty(\Tn)$,
the Fourier transform is extended uniquely to mapping
$\FTT: \Dcal'(\Tn)\to \Scal'(\Zn)$
by the formula
\begin{equation}\label{EQ:torus-FTT-inv}
   \jp{\FTT u,\va} := \jp{u,\iota\circ\FT_\Tn^{-1}\va},
\end{equation}
where $u\in \Dcal'(\Tn)$, $\va\in\Scal(\Zn)$, and
$\iota$ is defined by $(\iota\circ\psi)(x)=\psi(-x)$.
One can easily check that extension \eqref{EQ:torus-FTT-inv}
when restricted to $C^\infty(\Tn)$, is compatible with the 
definition \eqref{EQ:torus-FTT}.
Here, inclusion $C^\infty(\Tn)\subset\Dcal'(\Tn)$
is interpreted by
$\langle u,\va\rangle = \int_{\Tn} u(x)\ \va(x) \dslash x$.

It can be also easily seen that the Fourier analysis on the
$N$-inflated torus $N\Tn=\p{\R/2\pi N\Z}^n$
can be obtained from the Fourier
analysis of the standard torus $\Tn$ by changes of variables.
We can identify the dual group
$\widehat{N\Tn}$ of $N\Tn$ with $\frac{1}{N}\Zn$.
On $N\Tn$, we can define the Fourier transform 
$\FT_{N\Tn}:C^\infty(N\Tn)\to \Scal(\frac{1}{N}\Zn)$
by
$$
  (\FT_{N\Tn} g)(\eta) 
  := \int_{N\Tn} g(y) \erm^{-\irm y\cdot\eta}
  \ \dslash y, \;\; \eta\in\frac{1}{N}\Zn,
$$
where $\dslash y = (2\pi)^{-n}\ {\rm d}y$. 
Let us denote $g_N(x)=g(N x)$.
Then it is easy to see that we have the relation
\begin{equation}\label{EQ:tori-inflated-FTs}
\begin{aligned}
 \FT_{N\Tn} g(\eta) & =
 N^n \p{\FT_{\Tn} g_N}(N\eta), \;\;\ \eta\in\frac1N\Zn.
\end{aligned}
\end{equation}

\section{Difference calculus and discrete Taylor's theorem}
\label{SEC:diff-Taylor}

In this section we develop the discrete calculus which will
be needed in the sequel.
In particular, we will formulate and prove 
a discrete version of the Taylor 
expansion formula on the lattice $\Zn$.

Let $\sigma:\Zn\to\Bbb C$.
Let $e_j\in\Bbb N^n$, $(e_j)_j=1$ and $(e_j)_i = 0$
if $i\not=j$.
Define the partial difference operator $\triangle_{\xi_j}$ by
$$
  \triangle_{\xi_j} \sigma(\xi) := \sigma(\xi+e_j)-\sigma(\xi),
$$
and define
$$
  \triangle_\xi^{\alpha} = \triangle_{\xi_1}^{\alpha_1}
        \cdots \triangle_{\xi_n}^{\alpha_n}
$$
for $\alpha \in\Bbb N_0^n$.
It can be easily checked that 
these difference operators commute, i.e. that
$\triangle_\xi^\alpha\triangle_\xi^\beta=
\triangle_\xi^\beta \triangle_\xi^\alpha$ for all
multi-indices $\alpha,\beta\in\N_0^n$.
Let us summarize properties of differences:

\begin{prop}\label{PROP:torus-differences}
We have
$$
  \triangle_\xi^\alpha \sigma(\xi)
  = \sum_{\beta\leq\alpha} (-1)^{|\alpha-\beta|} {\alpha\choose\beta}
  \sigma(\xi+\beta).
$$
Let $\phi,\psi:\Zn\to\Bbb C$. Then we have the Leibnitz formula
for differences:
$$
  \triangle_\xi^\alpha(\phi \psi)(\xi) =
        \sum_{\beta\leq\alpha} {\alpha\choose \beta}
        \left(\triangle_\xi^\beta \phi(\xi)\right)
        \ \p{\triangle_\xi^{\alpha-\beta} \psi}(\xi+\beta).
$$
Also, the ``summation by parts'' is given by
\begin{equation}\label{EQ:summation-by-parts}
 \sum_{\xi\in\Zn} \phi(\xi)\ (\triangle_\xi^\alpha\psi)(\xi)
        = (-1)^{|\alpha|} \sum_{\xi\in\Zn} 
        (\overline{\triangle_\xi}^\alpha\phi)(\xi)
                \ \psi(\xi),
\end{equation}
where
$(\overline{\triangle_{\xi_j}}\phi)(\xi) = \phi(\xi)-\phi(\xi-e_j)$,
with the iterative definition for higher order differences. 
In \eqref{EQ:summation-by-parts} we assume that both series are
absolutely convergent. 
\end{prop}
Proposition \ref{PROP:torus-differences}
can be proved by induction and we omit the proof.
For $\theta\in\Zn$ and $\alpha\in\N_0^n$, we define
$\theta^{(\alpha)} = \theta_1^{(\alpha_1)}
\cdots\theta_n^{(\alpha_n)}$,
where $\theta_j^{(0)} = 1$ and
\begin{equation}\label{EQ:torus-def-torpol}
   \theta_j^{(k+1)} = \theta_j^{(k)} (\theta_j-k)
  = \theta_j (\theta_j-1)\ldots(\theta_j-k).
\end{equation} 
Then we have $
  \triangle_\theta^\gamma \theta^{(\alpha)}
        = \alpha^{(\gamma)}\ \theta^{(\alpha-\gamma)},
$
in analogy to the Euclidean case when
$\partial_\theta^\gamma \theta^\alpha= 
\alpha^{(\gamma)}\theta^{\alpha-\gamma}$.

For $b \geq 0$, let us denote
\begin{equation}\label{sumintegral}
  I_k^b := \sum_{0\leq k < b}\quad {\rm and}\quad
  I_k^{-b} := -\sum_{-b\leq k < 0}.
\end{equation}
One can think of $I_\xi^\theta\cdots$
as a discrete version of the one-dimensional integral
$\int_0^\theta\cdots {\rm d}\xi$;
in this discrete context,
the difference $\triangle_\xi$ takes the role of
the differential operator ${\rm d}/{\rm d}\xi$.

In the sequel,
we adopt the notational conventions
$$
  I_{k_1}^\theta I_{k_2}^{k_1} \cdots I_{k_\alpha}^{k_{\alpha-1}} 1 =
  \begin{cases}
    1, & {\rm if}\ \alpha = 0,\\
    I_{k_1}^\theta 1 , & {\rm if}\ \alpha = 1, \\
    I_{k_1}^\theta I_{k_2}^{k_1} 1, & {\rm if}\ \alpha = 2,
  \end{cases}
$$
and so on.

\begin{lemma}\label{l:integrals-of one}
If $\theta\in\Bbb Z$ and $\alpha\in\Bbb N_0$ then
\begin{equation}\label{EQ:integrals-of one}
  I_{k_1}^\theta I_{k_2}^{k_1} \cdots I_{k_\alpha}^{k_{\alpha-1}} 1
  = \frac{1}{\alpha!}\ \theta^{(\alpha)}.
\end{equation}
\end{lemma}

\begin{proof}
We observe simple equalities $k^{(0)} \equiv 1$,
$\triangle_k k^{(i)} = i\ k^{(i-1)}$
and $I_k^b \triangle_k k^{(i)} = b^{(i)}$, from
which \eqref{EQ:integrals-of one} follows by induction.
\end{proof}

We note that Lemma \ref{l:integrals-of one}
can be viewed as a discrete trivial version
of the fundamental theorem of calculus:
$\int_0^\theta f'(\xi)\ {\rm d}\xi = f(\theta)-f(0)$
for smooth enough $f:\Bbb R\to\Bbb C$
corresponds to
$I_\xi^\theta \triangle_\xi f(\xi) = f(\theta)-f(0)$
for $f:\Bbb Z\to\Bbb C$.
Lemma \ref{l:integrals-of one} immediately implies its
multidimensional version:
\begin{corollary}
If $\theta\in\Zn$ and $\alpha\in\Bbb N_0^n$ then
\begin{equation}\label{counting}
  \prod_{j=1}^n I_{k(j,1)}^{\theta_j} I_{k(j,2)}^{k(j,1)}
  \cdots I_{k(j,\alpha_j)}^{k(j,\alpha_j-1)} 1
  = \frac{1}{\alpha!}\ \theta^{(\alpha)},
\end{equation}
where $\prod_{j=1}^n I_j$ means $I_1 I_2\cdots I_n$, where
$I_j := I_{k(j,1)}^{\theta_j} I_{k(j,2)}^{k(j,1)}
\cdots I_{k(j,\alpha_j)}^{k(j,\alpha_j-1)}$.
\end{corollary}

Now we give the multidimensional version of 
the Taylor expansion formula.

\begin{theorem}[Discrete Taylor series on $\Zn$]
\label{THM:torus-Taylor-thm}
Let $p:\Zn\to\Bbb C$. Then we can write
$$
  p(\xi+\theta) =  \sum_{|\alpha| < M}
        \frac{1}{\alpha!}\ \theta^{(\alpha)} 
        \triangle_\xi^\alpha p(\xi) + r_M(\xi,\theta),
$$
with remainder $r_M(\xi,\theta)$ satisfying
\begin{equation}\label{remainderestimate}
  \left|\triangle_\xi^\omega r_M(\xi,\theta)\right| \leq
        C_M\ \abs{\theta^{(\alpha)}} 
        \max_{|\alpha|=M,\ \nu\in Q(\theta)} 
         \left| 
        \ \triangle_\xi^{\alpha+\omega} p(\xi+\nu) \right|,
\end{equation}
where $Q(\theta) := \{\nu\in\Zn:
        \ |\nu_j|\leq |\theta_j|
        \textrm{ for all } j=1,\ldots,n \}$.
\end{theorem}

\begin{proof}
For $0\not=\alpha\in\N_0^n$,
let us denote $m_\alpha:= \min\{j:\ \alpha_j\not=0\}$.
For $\theta\in\Zn$ and $i\in\{1,\ldots,n\}$,
let us define $\nu(\theta,i,k)\in\Zn$ by
$$
  \nu(\theta,i,k) := (\theta_1,\ldots,\theta_{i-1},k,0,\ldots,0),
$$
i.e.
$$
  \nu(\theta,i,k)_j =
  \begin{cases}
    \theta_j, & {\rm if}\ 1\leq j < i, \\
    k, & {\rm if}\ j = i, \\
    0, & {\rm if}\ i < j\leq n.
  \end{cases}
$$
We claim that the remainder can be written in the form
\begin{equation}
  r_M(\xi,\theta) = \sum_{|\alpha| = M} r_\alpha(\xi,\theta),
\end{equation}
where for each $\alpha$, we have
\begin{equation}\label{remainderinduction}
  r_\alpha(\xi,\theta)
  = \prod_{j=1}^n
  I_{k(j,1)}^{\theta_j}
  I_{k(j,2)}^{k(j,1)}\cdots I_{k(j,\alpha_j)}^{k(j,\alpha_j-1)}
    \ \triangle_\xi^\alpha p(\xi+\nu(\theta,m_\alpha,
      k(m_\alpha,\alpha_{m_\alpha})));
\end{equation}
recall (\ref{sumintegral}) and (\ref{counting}).
The proof of (\ref{remainderinduction}) is by induction.
The first remainder term $r_1$ is of the claimed form, since
$$
  r_1(\xi,\theta)
  = p(\xi+\theta)-p(\xi)
  = \sum_{i=1}^n r_{e_i}(\xi,\theta),
$$
where
$$
  r_{e_i}(\xi,\theta) = I_{k}^{\theta_i} \triangle_\xi^{e_i}
  p(\xi+\nu(\theta,i,k));
$$
here $r_{e_i}$ is of the form (\ref{remainderinduction})
for $\alpha = e_i$, $m(\alpha) = i$ and $\alpha_{m_\alpha} = 1$.
So suppose that the claim $(\ref{remainderinduction})$
is true up to order $|\alpha|=M$.
Then
\begin{eqnarray*}
  r_{M+1}(\xi,\theta)
  & = & r_M(\xi,\theta) - \sum_{|\alpha| = M} \frac{1}{\alpha!}
  \ \theta^{(\alpha)}\ \triangle_\xi^\alpha p(\xi) \\
  & = & \sum_{|\alpha| = M} \left( r_\alpha(\xi,\theta) - \frac{1}{\alpha!}
  \ \theta^{(\alpha)}\ \triangle_\xi^\alpha p(\xi) \right) \\
  & = &
    \sum_{|\alpha| = M} \prod_{j=1}^n
  I_{k(j,1)}^{\theta_j}
  I_{k(j,2)}^{k(j,1)}\cdots I_{k(j,\alpha_j)}^{k(j,\alpha_j-1)}
    \\
  & & \triangle_\xi^\alpha \left[ p(\xi+\nu(\theta,m_\alpha,
      k(m_\alpha,\alpha_{m_\alpha}))) - p(\xi) \right],
\end{eqnarray*}
where we used (\ref{remainderinduction}) and
(\ref{counting}) to obtain the last equality.
Combining this with the equality
$$
  p(\xi+\nu(\theta,m_\alpha,
      k)) - p(\xi)
    = \sum_{i=1}^{m_\alpha}
 I_{\ell}^{\nu(\theta,m_\alpha,k)_i}
    \ \triangle_\xi^{v_i} p(\xi+\nu(\theta,i,\ell)),
$$
we get
\begin{eqnarray*}
  r_{M+1}(\xi,\theta)
  & = &
    \sum_{|\alpha| = M} \prod_{j=1}^n
  I_{k(j,1)}^{\theta_j}
  I_{k(j,2)}^{k(j,1)}\cdots I_{k(j,\alpha_j)}^{k(j,\alpha_j-1)}
    \sum_{i=1}^{m_\alpha}
    I_{\ell(i)}^{\nu(\theta,m_\alpha,k(m_\alpha,\alpha_{m_\alpha}))_i} \\
  & & 
  \triangle_\xi^{\alpha+e_i} p(\xi+\nu(\theta,i,\ell(i))) \\
  & = &
  \sum_{|\beta| = M+1}
  \prod_{j=1}^n
  I_{k(j,1)}^{\theta_j}
  I_{k(j,2)}^{k(j,1)}\cdots 
  I_{k(j,\beta_j)}^{k(j,\beta_j-1)}  
  \triangle_\xi^\beta p(\xi+\nu(\theta,m_\beta,
      k(m_\beta,\beta_{m_\beta})));
\end{eqnarray*}
the last step here is just simple tedious book-keeping.
Thus the induction proof of (\ref{remainderinduction}) is complete.
Finally, let us prove estimate (\ref{remainderestimate}).
By (\ref{remainderinduction}), we obtain
\begin{eqnarray*}
  & & \left| \triangle_\xi^\omega r_M(\xi,\theta) \right| = \\
  & = & \left| \sum_{|\alpha|=M} \triangle_\xi^\omega r_\alpha(\xi,\theta)
  \right|\\
  & = & \left| \sum_{|\alpha| = M}
  \prod_{j=1}^n
  I_{k(j,1)}^{\theta_j}
  I_{k(j,2)}^{k(j,1)}\cdots I_{k(j,\alpha_j)}^{k(j,\alpha_j-1)} 
  \triangle_\xi^{\alpha+\omega} p(\xi+\nu(\theta,m_\alpha,
      k(m_\alpha,\alpha_{m_\alpha}))) \right| \\
  & \leq & \sum_{|\alpha| = M}\frac{1}{\alpha!}\left|\theta^{(\alpha)}\right|
  \max_{\nu\in Q(\theta)} \left|
     \triangle_\xi^{\alpha+\omega} p(\xi+\nu) \right|,
\end{eqnarray*}
where in the last step we used (\ref{counting}).
The proof is complete.
\end{proof}

\begin{rem}
In Theorem \ref{THM:torus-Taylor-thm} we estimated the 
remainder over the discrete box $Q(\theta)$, but an estimate over
a discrete path would have been enough.
\end{rem}

\section{Pseudo-differential operators and toroidal symbols}
\label{SEC:torus-symbols}

We note that given a continuous linear
operator $A:C^\infty(\Tn)\to C^\infty(\Tn)$,
we can consider its toroidal quantization
$$
  A\va(x) = \sum_{\xi\in\Zn} \int_{\Tn} {\rm e}^{{\rm i}(x-y)\cdot\xi}
  \ \sigma_A(x,\xi)\ f(y) \dslash y,
$$
where its {\it toroidal symbol}
$\sigma_A\in C^\infty(\Tn\times\Zn)$
is uniquely defined by the formula
$$
  \sigma_A(x,\xi) = {\rm e}^{-{\rm i}x\cdot\xi} Ae_\xi(x),
$$
where $e_\xi(x) := {\rm e}^{{\rm i}x\cdot\xi}$.
We note that for pseudo-differential operators
this would be just another quantization of the same class,
see Remark \ref{REM:mclean} for a precise statement.

For $\psi,\va\in C^\infty(\Tn)$,
let ${\psi\otimes\va}\in C^\infty(\Bbb T^{2n})$ be defined by
$\p{\psi\otimes\va}(x,y):=\psi(x)\va(y)$.
If $A:C^\infty(\Tn)\to\Dcal'(\Tn)$ is a 
continuous linear operator
then one can verify that
$$
  \langle K_A,\psi\otimes\va\rangle := \langle A\va,\psi\rangle
$$
defines the periodic Schwarz distributional kernel
$K_A\in\Dcal'(\Bbb T^{2n})$
of operator $A\in{\Lcal}(C^\infty(\Tn),\Dcal'(\Tn))$;
a common informal notation is
$$
  A\va(x) = \int_{\Tn} K_A(x,y)\ \va(y) \dslash y.
$$
The {convolution kernel} $k_A\in\Dcal'(\Bbb T^{2n})$ of $A$
is related to the Schwartz kernel by
$K_A(x,y)=k_A(x,x-y)$, i.e. we have
$$
  A\va(x) = \int_{\Tn} k_A(x,x-y)\ \va(y) \dslash y
$$
in the sense of distributions.
Notice that if $k_A(x)(y)=k_A(x,y)$ then
$$
  \FTT(k_A(x))(\xi) = \sigma_A(x,\xi).
$$

Let $m\in\Bbb R$, $0\leq\delta,\rho\leq 1$.
Then the {\em toroidal symbol class}
$S^m_{\rho,\delta}(\Tn\times\Zn)$
consists of those functions
$a(x,\xi)$ which are smooth in $x$ for all $\xi\in\Zn$,
and which satisfy
\begin{equation*}
  \left|\triangle_\xi^\alpha \partial_x^\beta a(x,\xi) \right|
        \leq C_{a\alpha\beta m}
                \ \langle\xi\rangle^{m-\rho|\alpha|+\delta|\beta|}
\end{equation*}
for every $x\in\Tn$, for every $\alpha,\beta\in\N_0^n$,
and for all $\xi\in\Zn$.
The class $S^m_{1,0}(\Tn\times\Zn)$ 
will be often denoted by writing simply $S^m(\Tn\times\Zn)$.

If $a\in S^m_{\rho,\delta}(\Tn\times\Zn)$,
we denote by $a(X,D)\in{\rm Op}S^m_{\rho,\delta}
(\Tn\times\Zn)$ the corresponding toroidal
pseudo-differential operator defined by
\begin{equation}\label{EQ:torus-pseudo-def}
  a(X,D)f(x)= \sum_{\xi\in\Zn} \int_{\Tn} 
   \ {\rm e}^{{\rm i}(x-y)\cdot\xi}\ a(x,\xi)\ f(y) 
         \dslash y.
\end{equation}
Naturally, $\sigma_{a(X,D)}(x,\xi)=a(x,\xi)$.

To contrast this with Euclidean (H\"ormander's) symbol classes,
we will write $b\in S^m_{\rho,\delta}(\Rn\times\Rn)$ if
$b\in C^\infty(\Rn\times\Rn)$ and if
\begin{equation*}
  \left|\partial_\xi^\alpha \partial_x^\beta b(x,\xi) \right|
        \leq C_{b\alpha\beta m}
                \ \langle\xi\rangle^{m-\rho|\alpha|+\delta|\beta|}
\end{equation*}
for every $x\in\Rn$, for every $\alpha,\beta\in\N_0^n$,
and for all $\xi\in\Rn$. If in addition
$b(\cdot,\xi)$ is $2\pi$-periodic
for every $\xi$ we will write 
$b\in S^m_{\rho,\delta}(\Tn\times\Rn)$. The corresponding
(Euclidean) pseudo-differential operator is then given by
\begin{equation*}\label{EQ:torus-pseudo-def-Rn}
  b(X,D)f(x)= \int_\Rn \int_{\Tn} 
   \ {\rm e}^{{\rm i}(x-y)\cdot\xi}\ b(x,\xi)\ f(y) 
         \dslash y \drm\xi.
\end{equation*}

The class $S^m_{\rho,\delta}(\Tn\times\Tn\times\Zn)$ 
of {\em toroidal compound symbols} consists
of the functions $a(x,y,\xi)$ which are smooth in $x$ and $y$
for all $\xi\in\Zn$ and which satisfy
\begin{equation}
  \left|\triangle_\xi^\alpha \partial_x^\beta \partial_y^\gamma
        a(x,y,\xi) \right|
        \leq C_{a\alpha\beta\gamma m}
   \ \langle\xi\rangle^{m-\rho|\alpha|+\delta|\beta+\gamma|}
\end{equation}
for every $x,y\in\Tn$, for every 
$\alpha,\beta,\gamma\in\N_0^n$, and for all $\xi\in\Zn$.
Such a function $a$ will be also called a 
{\it toroidal amplitude} of order $m\in\Bbb R$
of type $(\rho,\delta)$.
Formally we may also define
$$
  ({\rm Op}(a)f)(x) := \sum_{\xi\in\Zn} \int_{\Tn} 
    {\rm e}^{{\rm i}(x-y)\cdot\xi}\ a(x,y,\xi)\ f(y)  \dslash y
$$
for $f\in C^\infty(\Tn)$.

\begin{rem}\label{REM:mclean}
On $\Tn$,
H\"ormander's usual $(\rho,\delta)$ class of 
pseudo-differential operators 
${\rm Op}S^m_{\rho,\delta}(\Rn\times\Rn)$
of order $m\in\Bbb R$ which are $2\pi$-periodic in $x$ 
coincides with the class 
${\rm Op}S^m_{\rho,\delta}(\Tn\times\Zn)$, i.e.
$$
{\rm Op}S^m_{\rho,\delta}(\Tn\times\Rn)=
{\rm Op}S^m_{\rho,\delta}(\Tn\times\Zn),
$$
see e.g. \cite{McLean}, but in Theorem 
\ref{THM:torus-extend-symbols} we give a precise relation 
between actual symbols. The relation between the corresponding
operators is then given in Theorem \ref{THM:torus-per-symb}. 
\end{rem}

Now we will discuss the calculus of 
pseudo-differential operators with toroidal
symbols. For this, let us fix the notation first.
We define 
\begin{equation}\label{EQ:torus-def-toryder}
D_y^{(\alpha)}
= D_{y_1}^{(\alpha_1)}\cdots D_{y_n}^{(\alpha_n)},
\end{equation} 
where
$D_{y_j}^{(0)} = I$ and
$$
  D_{y_j}^{(k+1)}
  = D_{y_j}^{(k)}\left(\frac{\partial}{{\rm i}\partial y_j} - 
    k I\right)
  = \frac{\partial}{{\rm i}\partial y_j}
  \left(\frac{\partial}{{\rm i}\partial y_j} - I\right)
  \cdots \left(\frac{\partial}{{\rm i}\partial y_j} - k I\right).
$$

\begin{theorem}[Compound symbols]
Let $0\leq\delta<\rho\leq 1$.
For every amplitude
$a\in S^m_{\rho,\delta}(\Tn\times\Tn\times\Zn)$
there exists a unique symbol
$\sigma\in S^m_{\rho,\delta}(\Tn\times\Zn)$
satisfying
${\rm Op}(a)={\rm Op}(\sigma)$,
where
\begin{equation}
  \sigma(x,\xi) \sim \sum_{\alpha\geq 0} \frac{1}{\alpha!}
        \ \triangle_\xi^\alpha
        \ D_y^{(\alpha)} a(x,y,\xi)|_{y=x}.
\end{equation}
\end{theorem}

The proof of this theorem is essentially the 
same as that of the one-dimensional case of
$S^m_{1,0}(\Bbb T\times\Z)$ analyzed in 
\cite{TurunenVainikko}.
In \cite{TurunenVainikko} also the following results
were proved for symbols in $S^m_{1,0}(\Bbb T\times\Z)$
(see also \cite{SaranenVainikko}), and below
we give their extensions to $(\rho,\delta)$ classes in
$n$ dimensions. The possibility for such
an extension comes from the discrete Taylor
expansion and estimates for remainders
in Theorem \ref{THM:torus-Taylor-thm}. We omit these lengthy
proofs and refer to e.g. Theorem \ref{THM:torus-TP}
for a more general version of this type of arguments.

\begin{theorem}[Composition]\label{THM:torus-compositions}
Let $0\leq\delta<\rho\leq 1$ and let
$\sigma_A\in S^{m_1}_{\rho,\delta}(\Tn\times\Zn)$ and
$\sigma_B\in S^{m_2}_{\rho,\delta}(\Tn\times\Zn)$.
Then $\sigma_{AB}\in S^{m_1+m_2}_{\rho,\delta}
(\Tn\times\Zn)$
has the following asymptotic expansion:
\begin{equation}\label{BAcomposition}
  \sigma_{AB}(x,\xi) \sim
  \sum_{\alpha\geq 0} \frac{1}{\alpha!}
  \left( \triangle_\xi^\alpha \sigma_A(x,\xi) \right)
  D_x^{(\alpha)} \sigma_B(x,\xi),
\end{equation}
In particular, the symbol of the commutator $[A,B]=AB-BA$
belongs to the class $S^{m_1+m_2-1}_{\rho,\delta}(\Tn\times\Zn)$.
\end{theorem}

\begin{theorem}[Adjoint]
Let $0\leq\delta<\rho\leq 1$.
Let $\sigma_A\in S^m_{\rho,\delta}(\Tn\times\Zn)$.
Then the symbol $\sigma_{A^\ast}\in 
S^m_{\rho,\delta}(\Tn\times\Zn)$
of the adjoint operator $A^\ast$
has the asymptotic expansion
\begin{equation}\index{asymptotic expansion!of adjoints}\label{Aast(H)n}
  \sigma_{A^\ast}(x,\xi)\sim \sum_{\alpha\geq 0} \frac{1}{\alpha!}
  \ \triangle_\xi^\alpha D_x^{(\alpha)} \overline{\sigma_A(x,\xi)}.
\end{equation}
\end{theorem}

From the asymptotic expansion for composition of pseudo-differential operators,
we get an expansion for a parametrix of an elliptic operator:

\begin{theorem}[Ellipticity]
Let $\sigma_A\in S^m(\Tn\times\Zn)$
be elliptic in the sense that
there exist constants $C_0>0$ and $N_0\in\N_0$ such that
\begin{equation}\label{elliptic}
  |\sigma_A(x,\xi)|
  \geq C_0\ \langle\xi\rangle^m
\end{equation}
for all $(x,\xi)\in\Tn\times\Zn$ for which $|\xi|\geq N_0$;
this is equivalent to assuming that
there exists $\sigma_B\in S^{-m}(\Tn\times\Zn)$
such that $I-BA,I-AB$ are smoothing.
Let
$
  A \sim \sum_{j=0}^\infty A_j,
$
where $\sigma_{A_j}\in S^{m-j}(\Tn\times\Zn)$.
Then
$
  B \sim \sum_{k=0}^\infty B_k,
$
where $B_k\in S^{-m-k}(\Tn\times\Zn)$
is such that
$\sigma_{B_0}(x,\xi)= 1/\sigma_{A_0}(x,\xi)$
for large enough $|\xi|$, and recursively
$$
  \sigma_{B_N}(x,\xi) = \frac{-1}{\sigma_{A_0}(x,\xi)}
  \sum_{k=0}^{N-1} \sum_{j=0}^{N-k}
  \sum_{|\gamma|=N-j-k}
        \frac{1}{\gamma!} \left[
          \triangle_\xi^\gamma \sigma_{A_j}(x,\xi) \right]
        D_x^{(\gamma)} \sigma_{B_k}(x,\xi).
$$
\end{theorem}

\section{Extension of toroidal symbols}\label{extending}
\label{SEC:ext-symbols}

It is often useful to extend toroidal symbols
from $\Tn\times\Zn$
to $\Tn\times\Rn$, ideally getting symbols in
H\"ormander's symbol classes.
The case of $n=1$ and $(\rho,\delta)=(1,0)$ was considered
in \cite{TurunenVainikko} and \cite{SaranenVainikko}.
This extension can be done with a suitable convolution
that respects the symbol inequalities.
In the following,
$\delta_{0,\xi}$ is the Kronecker delta at $0\in\Bbb Z^n$,
i.e. $\delta_{0,0}=1$, and $\delta_{0,\xi}=0$ if $\xi\not=0$.
First we prepare the following useful functions
$\theta,\phi_\alpha\in\Scal(\Rn)$:

\begin{lemma}\label{LEMMA:torus-ext-Meyer}
There exist functions $\phi_\alpha\in\Scal(\Rn)$
(for each $\alpha\in\N_0^n$) and a function $\theta\in\Scal(\Rn)$ such that
$$
  {\Pcal}\theta(x):=\sum_{k\in\Zn}\theta(x+2\pi k)\equiv 1,
$$
$(\FTR{\theta})|_{\Zn}(\xi) = \delta_{0,\xi}$ and
$\partial_\xi^\alpha(\FTR{\theta})(\xi)
= \overline{\triangle}_\xi^\alpha \phi_\alpha(\xi)$
for all $\xi\in\Zn$.
\end{lemma}

\begin{proof}
Let us first consider the $1$-dimensional case.
Let $\theta=\theta_1\in C^\infty(\Bbb R^{1})$ such that
$$
  {\rm supp}(\theta_1) \subset (-2\pi,2\pi),\quad
  \theta_1(-x) = \theta_1(x),\quad
  \theta_1(2\pi-y) + \theta_1(y) = 1
$$
for $x\in\Bbb R$ and
for $0\leq y\leq 2\pi$;
these assumptions for $\theta_1$ are enough for us,
and of course the choice is not unique.
In any case, ${\theta_1} \in \Scal(\Bbb R^{1})$, so that
also $\FT_{\R} {\theta_1} \in \Scal(\Bbb R^{1})$.
If $\xi\in\Bbb Z^1$ then we have
\begin{eqnarray*}
  \FT_\R{\theta_1}(\xi)
   = 
  \int_{\Bbb R^{1}} \theta_1(x)\ {\rm e}^{-{\rm i}x\cdot\xi}
  \dslash x 
   = 
  \int_0^{2\pi} \left(\theta_1(x-2\pi) + \theta_1(x) \right)
  \ {\rm e}^{-{\rm i}x\cdot\xi} \dslash x  = 
  \delta_{0,\xi}.
\end{eqnarray*}
If a desired $\phi_\alpha\in\Scal(\Bbb R^1)$ exists,
it must satisfy
$$
  \int_{\Bbb R^{1}} {\rm e}^{{\rm i}x\cdot\xi}\ 
   \partial_\xi^\alpha \p{\FT_{\R}{\theta_1}}(\xi)
  \ {\rm d}\xi =
  \int_{\Bbb R^{1}} {\rm e}^{{\rm i}x\cdot\xi}
  \ \overline{\triangle}_\xi^\alpha \phi_\alpha(\xi)
  \ {\rm d}\xi =
  \p{1-\erm^{\irm x}}^\alpha  
  \int_{\Bbb R^{1}} {\rm e}^{{\rm i}x\cdot\xi}
   \phi_\alpha(\xi)  \ {\rm d}\xi
$$
due to the bijectivity of
$\FT_\R:\Scal(\Bbb R^{1})\to\Scal(\Bbb R^{1})$.
Integration by parts leads to the formula
$$
  (-{\rm i}x)^\alpha \theta_1(x) = (1-{\rm e}^{{\rm i}x})^\alpha
  (\FT_\R^{-1}\phi_\alpha)(x).
$$
Thus
$$
  (\FT_\R^{-1}\phi_\alpha)(x) =
  \left\{
  \begin{array}{rl}
     \left( \frac{-{\rm i}x}{1-{\rm e}^{{\rm i}x}} \right)^{\alpha}
  \theta_1(x),
                & \textrm{ if } 0<|x|<2\pi, \\ 
        1,      & \textrm{ if } x=0, \\ 
        0,  & \textrm{ if } |x|\geq 2\pi. \\ 
  \end{array}
  \right.
$$
The general $n$-dimensional case is reduced 
to the $1$-dimensional case,
since mapping
$\theta=(x\mapsto\theta_1(x_1)\theta_1(x_2)\cdots \theta_1(x_n))
\in\Scal(\Rn)$
has the desired properties.
\end{proof}

\begin{theorem}\label{THM:torus-extend-symbols}
Let $0<\rho\leq 1$ and $0\leq \delta\leq 1$.
Symbol 
$\widetilde{a}\in S^m_{\rho,\delta}(\Tn\times\Zn)$ is a toroidal
symbol if and only if
there exists an Euclidean symbol
$a\in S^m_{\rho,\delta}(\Tn\times\Rn)$
such that
$\widetilde{a} = a|_{\Tn\times\Zn}$. Moreover,
this extended symbol $a$ is unique modulo
$S^{-\infty}(\Tn\times\Rn)$.
\end{theorem}
The relation between the corresponding pseudo-differential operators
will be given in Theorem \ref{THM:torus-per-symb}.
\begin{proof}
Let us first prove the ``if'' part.
Let $a\in S^m_{\rho,\delta}(\Rn\times\Rn)$, and in this
part we can actually allow any $\rho$ and
$\delta$, for example $0\leq\rho,\delta\leq 1$.
By the Lagrange Mean Value Theorem, if $|\alpha|=1$ then
\begin{eqnarray*}
  \triangle_\xi^\alpha \partial_x^\beta \widetilde{a}(x,\xi)
   =  \triangle_\xi^\alpha \partial_x^\beta a(x,\xi) 
   =  \partial_\xi^\alpha \partial_x^\beta a(x,\xi)|_{\xi=\eta},
\end{eqnarray*}
where $\eta$ is on the line from $\xi$ to $\xi+\alpha$.
For a general multi-index $\alpha\in\Bbb N_0^n$, we also have
$$
  \triangle_\xi^\alpha \partial_x^\beta \widetilde{a}(x,\xi)
  = \partial_\xi^\alpha \partial_x^\beta a(x,\xi)|_{\xi=\eta}
$$
for some
$\eta\in Q := [\xi_1,\xi_1+\alpha_1]
\times\cdots\times[\xi_n,\xi_n+\alpha_n]$.
This can be shown by induction. Indeed, let us write 
$\alpha=\delta+\gamma$ for some $\delta=e_j$. Then we can
calculate
\begin{align*}
 \triangle_\xi^\alpha\partial_x^\beta \widetilde{a}(x,\xi) &
 =  \triangle_\xi^\delta\p{\triangle_\xi^\gamma
    \partial_x^\beta \widetilde{a}}(x,\xi) 
 = \triangle_\xi^{e_j}\p{\partial_\xi^\gamma \partial_x^\beta 
  a(x,\xi)|_{\xi=\zeta}} \\
 & = \partial_\xi^\gamma \partial_x^\beta a(x,\zeta+e_j)-
   \partial_\xi^\gamma \partial_x^\beta a(x,\zeta) =
  \partial_\xi^\alpha \partial_x^\beta a(x,\xi)|_{\xi=\eta}
\end{align*}
for some $\zeta$ and $\eta$,
where we used the induction hypothesis in the first line.
Therefore, we can estimate
\begin{multline*}
  \left|\triangle_\xi^\alpha
  \partial_x^\beta\widetilde{a}(x,\xi)\right|
   = 
  \left|\partial_\xi^\alpha\partial_x^\beta 
  a(x,\xi)|_{\xi=\eta\in Q}
  \right|\leq 
  C_{\alpha\beta m}\ 
  \langle\eta\rangle^{m-\rho|\alpha|+\delta|\beta|} 
   \leq 
  C'_{\alpha\beta m}\ \langle\xi\rangle^{m-\rho|\alpha|+\delta|\beta|}.
\end{multline*}

Let us now prove the ``only if'' part. 
First we show uniqueness. 
Let $a,b\in S^m_{\rho,\delta}(\Tn\times\Rn)$
and assume that 
$a|_{\Tn\times\Zn} = b|_{\Tn\times\Zn}$.
Let $c=a-b$. Then $c\in S^m_{\rho,\delta}(\Tn\times\Rn)$ and
it satisfies $c|_{\Tn\times\Zn} = 0$.
If $\xi\in\Rn\setminus\Zn$,
choose $\eta\in\Zn$ that is the nearest 
point (or one of the nearest points) to $\xi$.
Then we have the first order Taylor expansion
\begin{eqnarray*}
  c(x,\xi)
  & = & c(x,\eta) +
  \sum_{\alpha:\ |\alpha|=1} r_\alpha (x,\xi,\xi-\eta)\ 
  (\xi-\eta)^\alpha \\
  & = &
  \sum_{\alpha:\ |\alpha|=1} r_\alpha (x,\xi,\xi-\eta)\ 
  (\xi-\eta)^\alpha,
\end{eqnarray*}
where
$$
  r_\alpha(x,\xi,\theta) = \int_{0}^{1} (1-t)
  \ \partial_\xi^\alpha c(x,\xi+t\theta)
  \ {\rm d}t.
$$
Hence we have
$|c(x,\xi)| \leq C\ \langle\xi\rangle^{m-\rho}$. Continuing
the argument inductively for $c$ and its derivatives, 
we obtain the uniqueness modulo $S^{-\infty}(\Tn\times\Rn)$.

Let us now show the existence.
Let $\theta\in\Scal(\Rn)$ be as in the Lemma 
\ref{LEMMA:torus-ext-Meyer}.
Define $a:\Tn\times\Rn\to\Bbb C$ by
$$
  a(x,\xi) := \sum_{\eta\in\Zn} (\FTR \theta)(\xi-\eta)
  \ \widetilde{a}(x,\eta).
$$
It is easy to see that $\widetilde{a}=a|_{\Tn\times\Zn}$.
Furthermore, we have
\begin{eqnarray*}
  \left|\partial_\xi^\alpha\partial_x^\beta a(x,\xi)\right|
  & = &
  \left| \sum_{\eta\in\Zn}
    \partial_\xi^\alpha (\FTR \theta)(\xi-\eta)
    \ \partial_x^\beta \widetilde{a}(x,\eta)
  \right| \\
  & = &
  \left| \sum_{\eta\in\Zn}
    \overline{\triangle}_\xi^\alpha\phi_\alpha(\xi-\eta)
    \ \partial_x^\beta \widetilde{a}(x,\eta)
  \right| \\
  & = &
  \left| \sum_{\eta\in\Zn}
    \phi_\alpha(\xi-\eta)
    \ {\triangle}_\eta^\alpha\partial_x^\beta\widetilde{a}(x,\eta)
    \ (-1)^{|\alpha|} \right| \\
  & \leq &
    \sum_{\eta\in\Zn}
    |\phi_\alpha(\xi-\eta)|
    \ C_{\alpha\beta m}\ 
    \langle\eta\rangle^{m-\rho|\alpha|+\delta|\beta|} \\
  & \leq &
    C_{\alpha\beta m}'\ 
    \langle\xi\rangle^{m-\rho|\alpha|+\delta|\beta|}
    \sum_{\eta\in\Zn} |\phi_\alpha(\eta)|
    \ \langle\eta\rangle^{\left|m-\rho|\alpha|+
    \delta|\beta|\right|} \\
  & \leq &
    C_{\alpha\beta m}''\ \langle\xi\rangle^{m-\rho|\alpha|+
    \delta|\beta|},
\end{eqnarray*}
where we used the summation by parts formula 
\eqref{EQ:summation-by-parts}.
In the last two lines we also used that $\phi_\alpha\in\Scal(\Rn)$
and also a simple fact that for $p>0$ we have
$\jp{\xi+\eta}^p\leq \jp{\xi}^p\jp{\eta}^p$ and
$\jp{\xi+\eta}^{-p} \jp{\eta}^{-p}\leq \jp{\xi}^{-p}$,
for all $\xi,\eta\in\Rn$.
Thus $a\in S^m_{\rho,\delta}(\Tn\times\Rn)$.
\end{proof}

\section{Periodization of pseudo-differential operators}
\label{SEC:periodization}

In this section we describe the relation between operators
with Euclidean and toroidal quantizations and between
operators corresponding to symbols $a(x,\xi)$ and
$\widetilde{a}=a|_{\Tn\times\Zn}$, given by 
the operator of periodization of functions.

First we state a
property of a pseudo-differential operator $a(X,D)$ to map 
the space $\Scal(\Rn)$ into itself, which
will be of importance. The 
following class will be sufficient for our purposes, and
the proof is straightforward.

\begin{prop} \label{PROP:torus-S-to-S}
Let $a=a(x,\xi)\in C^\infty(\Rn\times\Rn)$ and assume that
there exist $\epsilon>0$ and $N\in\R$ such that for every
$\alpha, \beta$ there are constants $C_{\alpha\beta}$ and
$M(\alpha,\beta)$ such that the estimate
$$
 \abs{\partial_x^\alpha\partial_\xi^\beta a(x,\xi)} \leq
 C_{\alpha\beta} \jp{x}^{N+(1-\epsilon)|\beta|}
 \jp{\xi}^{M(\alpha,\beta)}
$$
holds for all $x,\xi\in\Rn$. Then pseudo-differential 
operator $a(X,D)$ with symbol $a(x,\xi)$ is continuous
from $\Scal(\Rn)$ to $\Scal(\Rn)$.
\end{prop}

Before analyzing the relation between operators
with Euclidean and toroidal quantizations,
we will describe the periodization
operator that will be of importance for such analysis.

\begin{theorem}[Periodization]\label{THM:torus-Poisson}
The {\it periodization} $\Pcal f:\Rn\to\Bbb C$
of a function $f\in{\Scal}(\Rn)$ is defined by
\begin{equation}\label{periodization}
  \Pcal f(x) := \sum_{k\in\Zn} f(x+2\pi k).
\end{equation}
Then $\Pcal:\Scal(\Rn)\to C^\infty(\Tn)$ is surjective and
$
  \| \Pcal f \|_{L^1(\Tn)} \leq \| f \|_{L^1(\Rn)}.
$
Moreover, we have
\begin{equation}\label{EQ:EQ:torus-Per-Fcoeff}
\Pcal f(x) = 
\FT_{\Tn}^{-1}\left((\FT_{\Rn} f)|_{\Zn}\right)(x)
\end{equation}
and
\begin{equation}\label{EQ:torus-Poissonsum}
  \sum_{k\in\Zn} f(x+2\pi k)
  = \sum_{\xi\in\Zn} {\rm e}^{{\rm i}x\cdot\xi}
  \ \p{\FT_{\Rn}f}(\xi).
\end{equation}
In particular, for $x=0$ this is the Poisson summation formula.
\end{theorem}

We note that by Theorem \ref{THM:torus-Poisson}
we may extend the periodization to $L^1(\Rn)$,
and also this extension is surjective.
This is actually rather trivial compared to the smooth 
case of Theorem \ref{THM:torus-Poisson} because we can
find a preimage $f\in L^1(\Rn)$ of $g\in L^1(\Tn)$ under
the periodization mapping $\Pcal$ by for example setting
$f=g|_{[0,2\pi]^n}$ and $f=0$ otherwise.

\begin{proof}
Let us show the surjectivity
of $\Pcal:\Scal(\Rn)\to C^\infty(\Tn)$.
Let $\theta\in\Scal(\Rn)$ be a function defined in
Lemma \ref{LEMMA:torus-ext-Meyer}.
Then for any $g\in C^\infty(\Tn)$ it holds that
$$
  \Pcal (g \theta)(x)
  = \sum_{k\in\Zn} g(x+2\pi k)\ \theta(x+2\pi k)
  = g(x) \sum_{k\in\Zn} \theta(x+2\pi k)
  = g(x),
$$
where $g\theta$ is the product of $\theta$ with 
$2\pi\Zn$-periodic function $g$ on $\Rn$.
We omit the straightforward proofs for the other claims.
\end{proof}

We note that in the case of the $N$-inflated torus $N\Tn$
we can use the periodization operator $\Pcal_N$ instead
of $\Pcal$, where $\Pcal_N:\Scal(\Rn)\to C^\infty(N\Tn)$
can be defined by
\begin{equation}\label{EQ:torus-N-period}
  \p{\Pcal_N f}(x)=\FT_{N\Tn}^{-1}\p{\FTR f|_{\frac{1}{N}\Zn}}
  (x),\;\; x\in N\Tn.
\end{equation}

Let us now establish some basic properties of pseudo-differential
operators with respect to periodization. 
\begin{defn}
We will say that a function
$a:\Rn\times\Rn\to\C$ is $2\pi$-periodic (we will always
mean that it is periodic with respect to 
the first variable $x\in\Rn$)
if the function $x\mapsto a(x,\xi)$ is
$2\pi\Zn$-periodic for all $\xi$. 
\end{defn}
As before, we use tilde to denote
the restriction of $a\in C^\infty(\Rn\times\Rn)$ 
to $\Rn\times\Zn$. If $a(x,\xi)$ is
$2\pi$-periodic, we can view it as a function on $\Tn\times\Zn$, 
and we write $\widetilde{a}=a|_{\Tn\times\Zn}.$

\begin{theorem} \label{THM:torus-per-symb}
Let $a=a(x,\xi)\in C^\infty(\Rn\times\Rn)$ 
be $2\pi$-periodic with respect to $x$ for every $\xi\in\Rn$. 
Assume that for every
$\alpha, \beta\in\N_0^n$ there are constants $C_{\alpha\beta}$ and
$M(\alpha,\beta)$ such that estimate
$$
 \abs{\partial_x^\alpha\partial_\xi^\beta a(x,\xi)} \leq
 C_{\alpha\beta} \jp{\xi}^{M(\alpha,\beta)}
$$
holds for all $x,\xi\in\Rn$. 
Let $\widetilde{a}=a|_{\Tn\times\Zn}$.
Then 
\begin{equation}\label{EQ:torus-per-eqs}
\Pcal\circ a(X,D) f=\widetilde{a}(X,D)\circ \Pcal f
\end{equation} 
for all $f\in\Scal(\Rn)$.
\end{theorem}
Note that it is not important in this theorem
that $a$ is in a symbol class. 

\begin{rem}\label{REM:per-op-Ck}
Note that by
Proposition \ref{PROP:torus-S-to-S} both sides
of \eqref{EQ:torus-per-eqs} are 
well-defined functions in $C^\infty(\Tn)$.
Moreover, equality \eqref{EQ:torus-per-eqs} can be justified
for $f$ in larger classes of functions. For example,
\eqref{EQ:torus-per-eqs} remains true if $f\in C_0^k(\Rn)$
is a $C^k$ compactly supported function for $k$ sufficiently large.
\end{rem}

\begin{proof}[Proof of Theorem \ref{THM:torus-per-symb}]
Let $f\in\Scal(\Rn)$. Then we have
\begin{eqnarray*}
  \Pcal(a(X,D)f)(x) & = & 
  \sum_{k\in\Zn} \p{a(X,D) f}(x+2\pi k) \\
  & = &
  \sum_{k\in\Zn} \int_{\Rn} {\rm e}^{{\rm i}(x+2\pi k)\cdot\xi}
  \ a(x+2\pi k,\xi)\ \p{\FTR f}(\xi) \ {\rm d}\xi \\
  & = &
  \int_{\Rn} \left(\sum_{k\in\Zn} {\rm e}^{{\rm i}2\pi k\cdot\xi}
  \right) {\rm e}^{{\rm i}x\cdot\xi}
  \ a(x,\xi)\ \p{\FTR f}(\xi) \ {\rm d}\xi \\
  & = &
  \int_{\Rn} \ \delta_{\Zn}(\xi)
  \erm^{\irm x\cdot\xi}\ a(x,\xi)\ \p{\FTR f}(\xi)
   \ {\rm d}\xi \\
  & = &
  \sum_{\xi\in\Zn} {\rm e}^{{\rm i}x\cdot\xi}
  \ a(x,\xi)\ \p{\FTR f}(\xi)  \\
  & = &
  \sum_{\xi\in\Zn} {\rm e}^{{\rm i}x\cdot\xi}
  \ a(x,\xi)\ \FTT(\Pcal f)(\xi) \\
  & = &
  \widetilde{a}(X,D)(\Pcal f)(x).
\end{eqnarray*}
As usual, these calculations can be justified in the
sense of distributions. 
\end{proof}

Let us now formulate a useful corollary of Theorem 
\ref{THM:torus-per-symb} that will be of importance later,
in particular in the proof of Theorem \ref{THM:torus-PT}.
If in Theorem \ref{THM:torus-per-symb}
we take function $f$ such that $f=g|_{[0,2\pi]^n}$ for some
$g\in C^\infty(\Tn)$, 
and $f=0$ otherwise, it follows immediately that 
$g=\Pcal f$. Adjusting this argument by shifting the
cube $[0,2\pi]^n$ if necessary, we obtain

\begin{corollary} \label{COR:torus-period-restr}
Let $a=a(x,\xi)$ be as in Theorem \ref{THM:torus-per-symb},
let $g\in C^\infty(\Tn)$, and let $V$ be an open cube in $\Rn$ 
with side length equal to $2\pi$. Assume that the support 
of $g|_{\overline{V}}$ is
contained in $V$. Then we have the equality
$$
  \widetilde{a}(X,D)g=\Pcal\circ a(X,D) (g|_{V}),
$$
where $g|_V:\Rn\to\C$ 
is defined as the restriction of $g$ to $V$ and
equal to zero outside of $V$.
\end{corollary}

Since we do not always have periodic symbols 
on $\Rn$ it
may be convenient to periodize them. If $a(X,D)$ is a
pseudo-differential operator with symbol $a(x,\xi)$, 
by $(\Pcal a)(X,D)$ we will denote the pseudo-differential
operator with symbol 
$$(\Pcal a)(x,\xi)=\sum_{k\in\Zn} a(x+2\pi k,\xi).$$ 
This procedure makes sense if, for example,  
$a$ is in $L^1(\Rn)$ with respect to $x$.

In the following proposition we will assume that supports
of symbols and functions are contained in $[-\pi,\pi]^n$.
We note that this is not restrictive if these functions
are already 
compactly supported. Indeed, if supports of $a(\cdot,\xi)$
and $f$ are contained in some
compact set independent of $\xi$, 
we can find some $N\in\N$ such that
they are contained in $[-N\pi,N\pi]^n$, and then use the
analysis on the $N$-inflated torus, with periodization
operator $\Pcal_N$ instead of $\Pcal$, 
defined in \eqref{EQ:torus-N-period}.

\begin{prop}
Let $a=a(x,\xi)$ satisfy
$\supp a\subset [-\pi,\pi]^n\times\Rn$
and be such that for every
$\alpha, \beta\in\N_0^n$ there are constants $C_{\alpha\beta}$ and
$M(\alpha,\beta)$ such that the estimate
$$
 \abs{\partial_x^\alpha\partial_\xi^\beta a(x,\xi)} \leq
 C_{\alpha\beta} \jp{\xi}^{M(\alpha,\beta)}
$$
holds for all $x,\xi\in\Rn$. 
Then we have
$$
a(X,D)f=(\Pcal a)(X,D)f+Rf,
$$
for all $f\in C^\infty(\Rn)$ with
$\supp f\subset [-\pi,\pi]^n$. Operator
$R$ extends to a smoothing pseudo-differential operator
$R:\Dcal^\prime(\Rn)\to \Scal(\Rn)$.
\label{p:p2}
\end{prop}

\begin{proof}
By the definition we can write
$$
  (\Pcal a)(X,D)f(x)=\sum_{k\in\Zn}\int_{\Rn} 
  {\rm e}^{{\rm i}x\cdot\xi}
  \ a(x+2\pi k,\xi)\ \FTR f(\xi) \ {\rm d}\xi,
$$
and let us define $Rf:=a(X,D)f-(\Pcal a)(X,D)f.$
The assumption on the support of $a$ implies
that for every $x$ there is only one $k\in\Zn$ for which
$a(x+2\pi k,\xi)\not=0$, so the sum consists of only one
term. It follows that $Rf(x)=0$ for $x\in [-\pi,\pi]^n$,
because for such $x$ the non-zero term corresponds to $k=0$. 
Let now $x\in\Rn\setminus [-\pi,\pi]^n$. Since
$$
  Rf(x) = -\sum_{k\in\Zn, k\not=0} \int_{\Rn} \int_{\Rn}
  {\rm e}^{{\rm i}(x-y)\cdot\xi}
  \ a(x+2\pi k,\xi)\ f(y) \dslash y \drm \xi
$$
is just a single term and $|x-y|>0$ on
$\supp f$, we can integrate by
parts with respect to $\xi$ any number of times. This
implies that $R\in\Psi^{-\infty}$ and that $Rf$ decays
at infinity faster than any power. The proof is complete
since the same argument can be applied to the derivatives of
$Rf$. 
\end{proof}

Proposition \ref{p:p2} allows us to extend formula of Theorem
\ref{THM:torus-per-symb} to 
compact perturbations of periodic symbols. We will use
it when $a(X,D)$ is a sum of a constant coefficient operator
and an operator with compactly (in $x$) supported symbol.

\begin{corollary}\label{COR:pseudodifferentialperiodization}
Let $a(X,D)$ be an operator with symbol
$$a(x,\xi)=a_1(x,\xi)+a_0(x,\xi),$$
where $a_1$ is as in Theorem \ref{THM:torus-per-symb}, 
$a_1$ is $2\pi$-periodic in $x$ for every $\xi\in\Rn$, and 
$a_0$ is as in Proposition \ref{p:p2},
supported in $[-\pi,\pi]^n\times\Rn$. 
Define
$$\widetilde{b}(x,\xi):=
\widetilde{a_1}(x,\xi)+\widetilde{\Pcal a_0}(x,\xi),
\;\; x\in\Tn, \xi\in\Zn.$$
Then we have
\begin{equation}\label{EQ:periodizationformula}
  \Pcal\circ a(X,D) f=\widetilde{b}(X,D)\circ \Pcal f+
  \Pcal\circ Rf
\end{equation}
for all $f\in\Scal(\Rn)$, and 
operator $R$ extends to
$R:\Dcal^\prime(\Rn)\to\Scal(\Rn),$ so that
$\Pcal\circ R:\Dcal'(\Rn)\to C^\infty(\Tn)$.
Moreover, if $a_1,a_0\in S^m_{\rho,\delta}(\Rn\times\Rn)$,
then $\widetilde{b}\in S^m_{\rho,\delta}(\Tn\times\Zn)$.
\label{p:p3}
\end{corollary}  

\begin{rem}\label{REM:per-op-Ck2}
Recalling Remark~\ref{REM:per-op-Ck},
\eqref{EQ:periodizationformula} can be justified
for larger function classes,
e.g. for $f\in C_0^k(\Rn)$ for $k$ sufficiently large
(which will be of use in Section \ref{SEC:torus-WF}).
\end{rem}

\begin{proof}[Proof of Corollary \ref{COR:pseudodifferentialperiodization}]
By Proposition \ref{p:p2} we can write 
$$a(X,D)=a_1(X,D)+(\Pcal a_0)(X,D)+R,$$ 
with $R:\Dcal^\prime(\Rn)\to\Scal(\Rn).$ 
Let us define $b(x,\xi)=a_1(x,\xi)+(\Pcal a_0)(x,\xi)$,
so that $\widetilde{b}=b|_{\Tn\times\Zn}$.
The symbol $b$ is $2\pi$-periodic, hence
for the operator
$b(X,D)=a_1(X,D)+(\Pcal a_0)(X,D)$
by Theorem \ref{THM:torus-per-symb} we have
$$\Pcal\circ b(X,D)=\widetilde{b}(X,D)\circ \Pcal=
\widetilde{a_1}(X,D)\circ \Pcal+
\widetilde{\Pcal a_0}(X,D)\circ \Pcal.$$
Since $R:\Dcal^\prime(\Rn)\to\Scal(\Rn)$, we also have
$\Pcal\circ R:\Dcal^\prime(\Rn)\to C^\infty(\Tn).$ 
Finally, if $a_1,a_0\in S^m_{\rho,\delta}(\Rn\times\Rn)$,
then $\widetilde{b}\in S^m_{\rho,\delta}(\Tn\times\Zn)$
by Theorem \ref{THM:torus-extend-symbols}.
The proof is complete.
\end{proof}

\section{Toroidal wave front sets}
\label{SEC:torus-WF}

Here we shall briefly
study microlocal analysis not on the cotangent bundle
of the torus but on $\Tn\times\Zn$,
which is better suited for Fourier series representations.
Let us define mappings
\begin{eqnarray}
  \pi_\Rn:\Rn\setminus\{0\}\to\Bbb S^{n-1}, &&
  \pi_\Rn(\xi)=\xi/|\xi|,\\
  \pi_{\Tn\times\Rn}:\Tn\times(\Rn\setminus\{0\})\to\Tn\times\Bbb S^{n-1}, &&
  \pi_{\Tn\times\Rn}(x,\xi)=(x,\xi/|\xi|).
\end{eqnarray}

We say that $K\subset\Rn\setminus\{0\}$ is a cone in $\Rn$ if
$\xi\in K$ and $\lambda>0$ imply $\lambda\xi\in K$.
We say that $\Gamma\subset\Zn\setminus\{0\}$ is a
{\it discrete cone}
if $\Gamma=\Zn\cap K$ for some cone $K$ in $\Rn$;
moreover, if here $K$ is open then $\Gamma$ is called
and {\it open discrete cone}.
We omit the straightforward proof of the following result:
\begin{prop}
$\Gamma \subset\Zn\setminus\{0\}$ is a discrete
cone if and only if $\Gamma=\Zn\cap\pi_\Rn^{-1}(\pi_\Rn(\Gamma))$.
\end{prop}

Let $u\in\Dcal'(\Tn)$.
The {\it toroidal wave front set}
${\rm WF}_T(u)\subset\Tn\times(\Zn\setminus\{0\})$ is defined as follows:
we say that
$(x_0,\xi_0)\in\Tn\times(\Zn\setminus\{0\})$
does not belong to
${\rm WF}_T(u)$ if and only if
there exist $\chi\in C^\infty(\Tn)$
and an open discrete cone $\Gamma\subset\Zn\setminus\{0\}$
such that $\chi(x_0)\not=0$, $\xi_0\in \Gamma$ and
$$
  \forall N > 0\ \exists C_N < \infty\ \forall\xi\in 
  \Gamma:
  \ \left|\FTT(\chi u)(\xi)\right|
  \leq C_N \langle\xi\rangle^{-N};
$$
in such a case we say that $\FTT(\chi u)$ decays rapidly in $\Gamma$.

We say that a pseudo-differential operator
$A\in\Psi^m(\Tn\times\Zn)={\rm Op}S^m(\Tn\times\Zn)$ is
{\it elliptic at the point
$(x_0,\xi_0)\in{\Tn}\times(\Zn\setminus\{0\})$}
if its toroidal symbol $\sigma_A:\Tn\times\Zn\to\Bbb C$
satisfies
$$
  |\sigma_A(x_0,\xi)| \geq C\ \langle\xi\rangle^m
$$
for some constant $C > 0$ as $|\xi|\to\infty$,
where $\xi\in \Gamma$ and
$\Gamma\subset\Zn\setminus\{0\}$ is an open discrete cone
containing $\xi_0$.
Should $\xi\mapsto\sigma_A(x_0,\xi)$ be rapidly decaying in
an open discrete cone containing $\xi_0$
then $A$ is said to be {\it smoothing at $(x_0,\xi_0)$}.
The {\it toroidal characteristic set of
$A\in\Psi^m(\Tn\times\Zn)$} is
$$
  {\rm char}_T(A) := \{ (x_0,\xi_0)\in\Tn\times(\Zn\setminus\{0\}):
  \ A\ {\rm is\ not\ elliptic\ at}\ (x_0,\xi_0) \},
$$
and the {\it toroidal wave front set of $A$} is
$$
  {\rm WF}_T(A) := \{  (x_0,\xi_0)\in
  \Tn\times(\Zn\setminus\{0\}):
  \ A\ {\rm is\ not\ smoothing\ at}\ (x_0,\xi_0) \}.
$$
Now,
$$
  {\rm WF}_T(A) \cup {\rm char}_T(A) =
  \Tn\times(\Zn\setminus\{0\})
$$
because if $(x,\xi)\not\in {\rm char}_T(A)$, it means that $A$
is elliptic at $(x,\xi)$, and hence not smoothing.
It can be shown that ${\rm WF}_T(A)=\emptyset$
if and only if $A$ is smoothing,
i.e. maps $\Dcal'(\Tn)$ to $C^\infty(\Tn)$
(equivalently, the Schwartz kernel is $C^\infty$-smooth).

\begin{prop}
Let $A,B\in {\rm Op}S^m(\Tn\times\Zn)$. Then
$$
  {\rm WF}_T(AB) \subset {\rm WF}_T(A)\cap{\rm WF}_T(B).
$$
\end{prop}

\begin{proof}
By Theorem \ref{THM:torus-compositions} applied to
pseudo-differential operators $A$ and $B$
we notice that
the toroidal symbol of 
$AB\in {\rm Op}S^{2m}(\Tn\times\Zn)$ has an asymptotic expansion
\begin{eqnarray*}
  \sigma_{AB}(x,\xi)
  & \sim &
  \sum_{\alpha\geq 0} \frac{1}{\alpha!}
  \ \left( \triangle_\xi^\alpha \sigma_A(x,\xi) \right)
  \ D_x^{(\alpha)} \sigma_B(x,\xi) \\
  & \sim &
  \sum_{\alpha\geq 0} \frac{1}{\alpha!}
  \ \left( \partial_\xi^\alpha \sigma_A(x,\xi) \right)
  \ \partial_x^{\alpha} \sigma_B(x,\xi),
\end{eqnarray*}
where in the latter expansion we have used smooth extensions of
toroidal symbols.
This expansion tells that $AB$ is smoothing at $(x_0,\xi_0)$
if $A$ or $B$ is smoothing at $(x_0,\xi_0)$.
\end{proof}

The notion of
the toroidal wave front set is compatible with the action of
pseudo-differential operators:

\begin{prop}\label{PROP:torus-WF-pseudo}
Let $u\in\Dcal'(\Tn)$ and 
$A\in {\rm Op} S^m_{\rho,\delta}(\Tn\times\Zn)$,
where $0\leq\rho\leq1$, $0\leq\delta<1$. Then
$${\rm WF}_T(Au)\subset {\rm WF}_T(u).$$
Especially, if $\va\in C^\infty(\Tn)$ does not vanish,
then ${\rm WF}_T(\va u)={\rm WF}_T(u)$.
\end{prop}
We omit the proof which is straightforward.

We will not pursue the complete analysis of toroidal wave
front sets here because most of their properties can be
obtained from the known properties of the usual wave front
sets and the following relation, where ${\rm WF}(u)$ stands
for the usual H\"ormander's wave front set of a
distribution $u$.

\begin{theorem}
Let $u\in\Dcal'(\Tn)$. Then
$$
  {\rm WF}_T(u)=\p{\Tn\times\Zn} \cap {\rm WF}(u).
$$
\end{theorem}

\begin{proof}
Without loss of generality,
let $u\in C^k(\Tn)$ for some large $k$,
and let $\chi\in C_0^\infty(\Rn)$ such that
${\rm supp}(\chi)\subset (0,2\pi)^n$.
If $\FTR(\chi u)$ decays rapidly in an open cone $K\subset\Rn$
then $\FTT(\Pcal(\chi u))=\FTR(\chi u)|_{\Zn}$ decays rapidly
in the open discrete cone $\Zn\cap K$.
Hence ${\rm WF}_T(u) \subset (\Tn\times\Zn)\cap {\rm WF}(u)$.

Next, we need to show that
$(\Tn\times\Zn)\setminus {\rm WF}_T(u)
\subset (\Tn\times\Zn)\setminus {\rm WF}(u)$.
Let $(x_0,\xi_0)\in (\Tn\times\Zn)\setminus {\rm WF}_T(u)$
(where $\xi_0\not=0$). We must show that
$(x_0,\xi_0)\not\in {\rm WF}(u)$.
There exist $\chi\in C^\infty(\Tn)$
(we may assume that ${\rm supp}(\chi)\subset (0,2\pi)^n$ as above)
and an open cone $K\subset\Rn$ such that
$\chi(x_0)\not=0$, $\xi_0\in\Zn\cap K$ and that
$\FTT(\Pcal(\chi u))$ decays rapidly in $\Zn\cap K$.

Let $K_1\subset\Rn$ be an open cone such that
$\xi_0\in K_1\subset K$ and that
the closure $\overline{K_1}\subset K\cup\{0\}$.
Take any function $w\in C^\infty(\Bbb S^{n-1})$ such that
$$
  w(\omega) = \begin{cases}
    1, & {\rm if}\ \omega\in\Bbb S^{n-1}\cap K_1, \\
    0, & {\rm if}\ \omega\in\Bbb S^{n-1}\setminus K.
    \end{cases}
$$
Let $a\in C^\infty(\Rn\times\Rn)$ be independent
of $x$ and such that
$a(x,\xi) = w(\xi/|\xi|)$ whenever $|\xi| \geq 1$.
Then $a\in S^0(\Rn\times\Rn)$.
Let $\widetilde{a} = a|_{\Tn\times\Zn}$,
so that $\widetilde{a}\in S^0(\Tn\times\Zn)$
by Theorem \ref{THM:torus-extend-symbols}.

By Corollary~\ref{COR:pseudodifferentialperiodization}, we have
$\Pcal(\chi\ a(X,D)f)
= \Pcal(\chi)\ \widetilde{a}(X,D)(\Pcal f) + \Pcal(R f)$
for all Schwartz test functions $f$,
for a smoothing operator $R:\Dcal'(\Rn)\to\Scal(\Rn)$.
By Remark \ref{REM:per-op-Ck2} we also have
$\Pcal(\chi\ a(X,D)(\chi u))
= \Pcal(\chi)\ \widetilde{a}(X,D)(\Pcal(\chi u)) + \Pcal(R(\chi u))$,
where the right hand side belongs to $C^\infty(\Tn)$,
since its Fourier coefficients decay rapidly on the whole $\Zn$.
Therefore also $\Pcal(\chi\ a(X,D)(\chi u))$
belongs to $C^\infty(\Tn)$.

Thus
$\chi\ a(X,D)(\chi u)\in C_0^\infty(\Rn)$.
Let $\xi\in K_1$ such that $|\xi|\geq 1$. Then we have
$\FTR(a(X,D)(\chi u))(\xi) = w(\xi/|\xi|)\ \FTR(\chi u)(\xi)
= \FTR(\chi u)(\xi)$.
Thus $\FTR(\chi u)$ decays rapidly on $K_1$.
Therefore $(x_0,\xi_0)$ does not belong to ${\rm WF}(u)$.
\end{proof}

\section{Calculus of Fourier series operators}
\label{SEC:FSO-calculus}

In this section we consider analogues of Fourier integral 
operators on the torus $\Tn$. We will call such operators
Fourier series operators and study their composition
formulae with pseudo-differential operators 
on the torus. Fourier series operators will be
defined by the formula
\begin{equation}\label{EQ:torus-TP-defT0}
   Tu(x) := \sum_{\xi\in\Zn} \int_{\Tn}
  {\rm e}^{{\rm i}(\phi(x,\xi)-y\cdot\xi)}
  \ a(x,y,\xi)\ u(y)
  \dslash y,
\end{equation} 
where $a\in C^\infty(\Tn\times\Tn\times\Zn)$ is a 
toroidal amplitude and $\phi$ is a real-valued
phase function. We note that if $u\in C^\infty(\Tn)$,
for function $Tu$ to
be well-defined on the torus we need that the integral
\eqref{EQ:torus-TP-defT0} is $2\pi$-periodic in $x$.
Therefore, by identifying functions on $\Tn$ with
$2\pi$-periodic functions on $\Rn$, we will require that
the phase function $\phi:\Rn\times\Zn\to\R$ is such
that function $x\mapsto {\rm e}^{\irm\phi(x,\xi)}$ 
is $2\pi$-periodic for
all $\xi\in\Zn$. Then operator $T:C^\infty(\Tn)\to\Dcal'(\Tn)$
can be justified in the usual way for oscillatory integrals.

\begin{rem}
Assume that function $\phi:\Rn\times\Zn\to\R$ is 
in $C^k$ with respect to $x$ for all $\xi\in\Zn$.
Assume also that function $x\mapsto {\rm e}^{\irm\phi(x,\xi)}$ 
is $2\pi$-periodic for
all $\xi\in\Zn$. Then functions 
$x\mapsto\partial_x^\alpha\phi(x,\xi)$
are $2\pi$-periodic for all
$\xi\in\Zn$ and all $\alpha\in\N_0^n$ with 
$1\leq |\alpha|\leq k$.
\end{rem}
Composition formulae of this section can be compared with
those obtained by Ruzhansky and Sugimoto \cite{RS06, RS08a},
but the assumptions on the regularity or boundedness of
higher order derivatives of phases and amplitudes are 
redundant here thanks to the fact that $\xi\in\Zn$ takes
only discrete values.

We recall the notation for the
toroidal version of Taylor polynomials and 
the corresponding derivatives introduced in
\eqref{EQ:torus-def-torpol} and 
\eqref{EQ:torus-def-toryder}, which will be used in the
formulation of the following theorems. 
We also define 
\begin{equation}\label{EQ:torus-def-toryder-min}
(-D_y)^{(\alpha)}
= (-D_{y_1})^{(\alpha_1)}\cdots (-D_{y_n})^{(\alpha_n)},
\end{equation} 
where
$-D_{y_j}^{(0)} = I$ and
$$
  (-D_{y_j})^{(k+1)}
  = (-D_{y_j})^{(k)}\left(-\frac{\partial}{{\rm i}\partial y_j} - 
    k I\right)
  = -\frac{\partial}{{\rm i}\partial y_j}
  \left(-\frac{\partial}{{\rm i}\partial y_j} - I\right)
  \cdots \left(-\frac{\partial}{{\rm i}\partial y_j} - k I\right).
$$
We also recall
definition \eqref{EQ:torus-pseudo-def} of 
pseudo-differential operators
with toroidal symbols.

\begin{theorem}[Composition $T\circ P(X,D)$]\label{THM:torus-TP}
Let $\phi:\Rn\times\Zn\to\R$ be such that function
$x\mapsto \erm^{\irm\phi(x,\xi)}$ 
is $2\pi$-periodic for all $\xi\in\Zn$.
Let $T:C^\infty(\Tn)\to\Dcal'(\Tn)$ be defined by
\begin{equation}\label{EQ:torus-TP-defT}
   Tu(x) := \sum_{\xi\in\Zn} \int_{\Tn}
  {\rm e}^{{\rm i}(\phi(x,\xi)-y\cdot\xi)}
  \ a(x,y,\xi)\ u(y)
  \dslash y,
\end{equation}
where
the toroidal amplitude
$a\in C^\infty(\Tn\times\Tn\times\Zn)$
satisfies
$$
  \left| \partial_x^\alpha \partial_y^\beta a(x,y,\xi) \right|
  \leq C_{\alpha\beta m}\ \langle\xi\rangle^m
$$
for every $x,y\in\Tn$, $\xi\in\Zn$ and
$\alpha,\beta\in\Bbb N_0^n$.
Let $p\in S^\ell(\Tn\times\Zn)$ be a toroidal symbol and
$P=p(X,D)$ the corresponding pseudo-differential operator.
Then the composition $TP$ has the form
$$
  TPu(x) = \sum_{\xi\in\Zn} \int_{\Tn}
  {\rm e}^{{\rm i}(\phi(x,\xi)-z\cdot\xi)}
  \ c(x,z,\xi)\ u(z)
   \dslash z,
$$
where
$$
  c(x,z,\xi) = \sum_{\eta\in\Zn} \int_{\Tn}
  {\rm e}^{{\rm i}(y-z)\cdot(\eta-\xi)}
  \ a(x,y,\xi)\ p(y,\eta)
   \dslash y
$$
satisfies
$$
  \left| \partial_x^\alpha \partial_z^\beta c(x,z,\xi) \right|
  \leq C_{\alpha\beta m t}\ \langle\xi\rangle^{m+\ell}
$$
for every $x,z\in\Tn$, $\xi\in\Zn$ and
$\alpha,\beta\in\N_0^n$.
Moreover, we have an asymptotic expansion
$$
  c(x,z,\xi) \sim \sum_{\alpha\geq 0}
  \frac{1}{\alpha!}\ (-D_z)^{(\alpha)}
  \left[ a(x,z,\xi)\ \triangle_\xi^\alpha p(z,\xi)
  \right].
$$
Furthermore, if
$0\leq\delta<\rho\leq 1$,
$p\in S_{\rho,\delta}^\ell(\Tn\times\Zn)$ and
$a\in S_{\rho,\delta}^m(\Tn\times\Tn\times\Zn)$, then
$c\in S_{\rho,\delta}^{m+\ell}(\Tn\times\Tn\times\Zn)$.
\end{theorem}

\begin{rem}\label{REM:torus-pseudo-TP}
We note that if $T$ in \eqref{EQ:torus-TP-defT}
is a pseudo-differential operator
with phase $\phi(x,\xi)=x\cdot\xi$ and amplitude
$a(x,y,\xi)=a(x,\xi)$ independent of $y$, then the 
asymptotic expansion formula of two pseudo-differential 
operators $T\circ P$ becomes
$$
  c(x,z,\xi) \sim \sum_{\alpha\geq 0}
  \frac{1}{\alpha!}\ a(x,\xi)\ (-D_z)^{(\alpha)}
  \triangle_\xi^\alpha p(z,\xi).
$$
This is another representation for the composition
compared to Theorem \ref{THM:torus-compositions},
but the amplitude of the composition here is a
compound symbol.
\end{rem}

\begin{proof}
Let us calculate the composition $TP$:
\begin{eqnarray*}
  TPu(x)
  & = & \sum_{\xi\in\Zn} \int_{\Tn}
  {\rm e}^{{\rm i}(\phi(x,\xi)-y\cdot \xi)}\ a(x,y,\xi)\ Pu(y)
  \dslash y \\
  & = & \sum_{\xi\in\Zn} \int_{\Tn}
  {\rm e}^{{\rm i}(\phi(x,\xi)-y\cdot \xi)}\ a(x,y,\xi)
  \sum_{\eta\in\Zn} \int_{\Tn}
  {\rm e}^{{\rm i}(y-z)\cdot\eta}\ p(y,\eta)\ u(z)
  \dslash z  \dslash y \\
  & = & \sum_{\xi\in\Zn} \int_{\Tn}
  {\rm e}^{{\rm i}(\phi(x,\xi)-z\cdot\xi)}\ c(x,z,\xi)\ u(z)
  \dslash z,
\end{eqnarray*}
where
$$
  c(x,z,\xi)
  = \sum_{\eta\in\Zn} \int_{\Tn}
  {\rm e}^{{\rm i}(y-z)\cdot(\eta-\xi)}\ a(x,y,\xi)\ p(y,\eta)
  \dslash y.
$$
Denote $\theta := \eta-\xi$,
so that by the discrete Taylor formula as in Theorem 
\ref{THM:torus-Taylor-thm},
we formally get
\begin{eqnarray*}
  c(x,z,\xi)
  & \sim &
  \sum_{\theta\in\Zn} \int_{\Tn}
  {\rm e}^{{\rm i}(y-z)\cdot\theta}\ a(x,y,\xi)\ \sum_{\alpha\geq 0}
  \frac{1}{\alpha!}\ \theta^{(\alpha)}\ \triangle_\xi^\alpha p(y,\xi)
  \dslash y \\
  & = & \sum_{\alpha\geq 0} \frac{1}{\alpha!}
  \sum_{\theta\in\Zn} \int_{\Tn}
  \theta^{(\alpha)}\ {\rm e}^{{\rm i}(y-z)\cdot\theta}\ a(x,y,\xi)
  \ \triangle_\xi^{\alpha} p(y,\xi)
  \dslash y \\
  & = & \sum_{\alpha\geq 0} \frac{1}{\alpha!}
  \ (-D_y)^{(\alpha)} \left( a(x,y,\xi)
    \ \triangle_\xi^{\alpha} p(y,\xi) \right)|_{y=z}.
\end{eqnarray*}
Now we have to justify the asymptotic expansion.
First we take a discrete Taylor expansion and using
Theorem \ref{THM:torus-Taylor-thm} we obtain
$$
  p(y,\xi+\theta) = \sum_{|\omega| < M} 
  \frac{1}{\omega!}\ \theta^{(\omega)}
  \ \triangle_\xi^\omega p(y,\xi)
  + r_M(y,\xi,\theta).
$$
Let $Q(\theta) = \{ \nu\in\Zn:
\ |\nu_j|\leq |\theta_j| \textrm{ for all } j=1,\ldots,n\}$
as in Theorem \ref{THM:torus-Taylor-thm}.
Then
\begin{eqnarray*}
  \left| \triangle_\xi^\alpha \partial_y^\beta r_M(y,\xi,\theta)  \right|
  & \leq & C \ \langle\theta\rangle^M
  \ \max_{\substack{|\omega|=M\\ \nu\in Q(\theta)}}
  \left| \triangle_\xi^{\alpha+\omega}\partial_y^\beta p(y,\xi+\nu) \right|\\
  & \leq & C \ \langle\theta\rangle^{M}
  \ \max_{\nu\in Q(\theta)} \langle\xi+\nu\rangle^{\ell-|\alpha|-M} \\
  & \leq & C \ \langle\theta\rangle^{M}
  \ \max_{\nu\in Q(\theta)} \langle\nu\rangle^{\left|\ell-|\alpha|-M\right|}
  \ \langle\xi\rangle^{\ell-|\alpha|-M} \\
  & \leq & C \ \langle\theta\rangle^{2M+|\ell|+|\alpha|}
  \ \langle\xi\rangle^{\ell-|\alpha|-M}.
\end{eqnarray*}
The corresponding remainder term in the asymptotic expansion
of $c(x,z,\xi)$ is
\begin{eqnarray*}
  R_M(x,z,\xi)
  & = & \sum_{\theta\in\Zn} \int_{\Tn} {\rm e}^{{\rm i}(y-z)\cdot\theta}
  \ a(x,y,\xi)\ r_M(y,\xi,\theta) \dslash y \\
  & = & \sum_{\theta\in\Zn} \int_{\Tn}
  {\rm e}^{{\rm i}(y-z)\cdot\theta}\ \langle\theta\rangle^{-2N}
  \ (I-\Lap_y)^N \left[ a(x,y,\xi)\ r_M(y,\xi,\theta) \right]\dslash y,
\end{eqnarray*}
where we integrated by parts exploiting that
$(I-\Lap_y){\rm e}^{{\rm i}(y-z)\cdot\theta}
= \langle\theta\rangle^2\ {\rm e}^{{\rm i}(y-z)\cdot\theta}$,
where $\Lap_y$ is the Laplacian with respect to $y$.
Thus we get the estimate
\begin{eqnarray*}
  \left| \triangle_\xi^\alpha \partial_x^\beta \partial_z^\gamma R_M(x,z,\xi)
  \right| & \leq & C\ \sum_{\theta\in\Zn}
  \langle\theta\rangle^{|\gamma| - 2N + 2M + |\ell| + |\alpha|}
  \ \langle\xi\rangle^{m+\ell-M},
\end{eqnarray*}
and we take $N\in\Bbb N$ so large that
this sum over $\theta$ converges.
Hence
$$
  \left| \triangle_\xi^\alpha \partial_x^\beta \partial_z^\gamma
    R_M(x,z,\xi) \right|
  \leq C\ \langle\xi\rangle^{m+\ell-M}.
$$
This completes the proof of the first part of the theorem.
Finally, we assume that 
$a\in S^m_{\rho,\delta}(\Tn\times\Tn\times\Zn)$.
Then also the terms in the asymptotic expansion and 
the remainder $R_M$ have corresponding decay properties in 
the $\xi$-differences, leading to the amplitude 
$c\in S^m_{\rho,\delta}(\Tn\times\Tn\times\Zn)$.
This completes the proof.
\end{proof}

We now formulate the theorem about compositions of
operators in the different order.

\begin{theorem}[Composition $P(X,D)\circ T$]\label{THM:torus-PT}
Let $\phi:\Rn\times\Zn\to\R$ be such that function
$x\mapsto \erm^{\irm\phi(x,\xi)}$ 
is $2\pi$-periodic for
all $\xi\in\Zn$.
Let $T:C^\infty(\Tn)\to\Dcal'(\Tn)$ such that
$$
  Tu(x) := \sum_{\xi\in\Zn} \int_{\Tn}
  {\rm e}^{{\rm i}(\phi(x,\xi)-y\cdot\xi)}
  \ a(x,y,\xi)\ u(y)
  \dslash y,
$$
where $a\in C^\infty(\Tn\times\Tn\times\Zn)$ 
satisfies
$$
  \left|\partial_x^\alpha\partial_y^\beta a(x,y,\xi) \right|
  \leq C_{\alpha\beta m}\ \langle\xi\rangle^m
$$
for every $x,y\in\Tn$, $\xi\in\Zn$ and
$\alpha,\beta\in\N_0^n$.
Assume that
for some $C>0$ we have
\begin{equation}\label{EQ:torus-compPT-phi}
  C^{-1}\ \langle\xi\rangle
  \leq \langle\nabla_x\phi(x,\xi)\rangle
  \leq C\ \langle\xi\rangle
\end{equation}
for every $x\in\Tn$, $\xi\in\Zn$,
and that
\begin{equation}\label{EQ:torus-compPT-phi2}
  \left|\partial_x^\alpha\phi(x,\xi)\right| \leq C_\alpha\ \langle\xi\rangle,
  \quad
  \left|\partial_x^\alpha \triangle_\xi^\beta \phi(x,\xi) \right|
  \leq C_{\alpha\beta}
\end{equation}
for every $x\in\Tn$, $\xi\in\Zn$ and
$\alpha,\beta\in\N_0^n$ with $|\beta|=1$.
Let $\widetilde{p}\in S^\ell(\Tn\times\Zn)$ be a toroidal symbol and
let $P=\widetilde{p}(X,D)$ be the corresponding pseudo-differential 
operator. Let $p(x,\xi)$ denote an extension 
of $\widetilde{p}(x,\xi)$ to a symbol
in $S^\ell(\Tn\times\Rn)$ as given in Theorem
\ref{THM:torus-extend-symbols}.
Then
$$
  PTu(x) = \sum_{\xi\in\Zn} \int_{\Rn}
  {\rm e}^{{\rm i}(\phi(x,\xi)-z\cdot\xi)}
  \ c(x,z,\xi)\ u(z)
  \dslash z,
$$
where
$$
  \left| \partial_x^\alpha\partial_z^\beta c(x,z,\xi) \right|
  \leq C_{\alpha\beta}\ \langle\xi\rangle^{m+\ell}
$$
for every $x,z\in\Tn$, $\xi\in\Zn$ and
$\alpha,\beta\in\Bbb N_0^n$.
Moreover,
\begin{equation}\label{EQ:torus-PT-asympt}
 c(x,z,\xi) \sim \sum_{\alpha\geq 0}
  \frac{{\rm i}^{-|\alpha|}}{\alpha!}
    \ \partial_\eta^\alpha p(x,\eta)|_{\eta=\nabla_x\phi(x,\xi)}
    \ \partial_y^\alpha \left[
      {\rm e}^{{\rm i}\Psi(x,y,\xi)} a(y,z,\xi) \right] |_{y=x},
\end{equation} 
where
$$
  \Psi(x,y,\xi) := \phi(y,\xi) - \phi(x,\xi) + 
  (x-y)\cdot\nabla_x\phi(x,\xi).
$$
\end{theorem}
Let us make some remarks about quantities appearing in the
asymptotic extension formula \eqref{EQ:torus-PT-asympt}.
It is geometrically reasonable to evaluate the symbol 
$\widetilde{p}(x,\xi)$
at the real flow generated by the phase function $\phi$
of the Fourier series operator $T$. 
This is the main complication compared with pseudo-differential 
operators for which we have Proposition \ref{PROP:torus-WF-pseudo}.
However, although a-priori
symbol $\widetilde{p}$ is defined only
on $\Tn\times\Zn$, we can still extend it to
a symbol $p(x,\xi)$ on $\Tn\times\Rn$ by Theorem
\ref{THM:torus-extend-symbols}, so that 
the restriction 
$\partial_\eta^\alpha p(x,\eta)|_{\eta=\nabla_x\phi(x,\xi)}$
makes sense. We also note that function $\Psi(x,y,\xi)$
can not be in general considered as a function on
$\Tn\times\Tn\times\Zn$ because it may not be periodic in
$x$ and $y$. However, we can still observe that the derivatives
$\partial_y^\alpha \left[
      {\rm e}^{{\rm i}\Psi(x,y,\xi)} a(y,z,\xi) \right] |_{y=x}$
are periodic in $x$ and $z$, so all terms in the right hand side of
\eqref{EQ:torus-PT-asympt} are well-defined functions on
$\Tn\times\Tn\times\Zn$.

Finally, we note that if $\phi\in S^1_{\rho,\delta}(\Rn\times\Rn)$,
$\widetilde{p}\in S^\ell_{\rho,\delta}(\Tn\times\Zn)$,
$a\in S^m_{\rho,\delta}(\Tn\times\Tn\times\Zn)$, and
$0\leq\delta<\rho\leq 1$, then we also have
$c\in S^{m+\ell}_{\rho,\delta}(\Tn\times\Tn\times\Zn)$.
\begin{proof}
To simplify the notation, let us drop writing tilde on $p$,
and denote both symbols $\widetilde{p}$ and $p$ by the 
same letter $p$. There should be no confusion since they
coincide on $\Tn\times\Zn$.
Let $P=p(X,D)$.
We can write
\begin{eqnarray*}
  PTu(x)
  & = & \sum_{\eta\in\Zn} \int_{\Tn}
  {\rm e}^{{\rm i}(x-y)\cdot\eta}\ p(x,\eta)\ Tu(y)\dslash y \\
  & = & \sum_{\eta\in\Zn} \int_{\Tn}
  {\rm e}^{{\rm i}(x-y)\cdot\eta}\ p(x,\eta)
  \sum_{\xi\in\Zn} \int_{\Tn}
  {\rm e}^{{\rm i}(\phi(y,\xi)-z\cdot\xi)}\ a(y,z,\xi)\ u(z)
  \dslash z\dslash y \\
  & = & \sum_{\xi\in\Zn} \int_{\Tn}
  {\rm e}^{{\rm i}(\phi(x,\xi)-z\cdot\xi)}\ c(x,z,\xi)\ u(z)
  \dslash z,
\end{eqnarray*}
where
$$
  c(x,z,\xi) = \sum_{\eta\in\Zn} p(x,\eta) \int_{\Tn}
  {\rm e}^{{\rm i}(\phi(y,\xi)-\phi(x,\xi)+(x-y)\cdot\eta)}
  \ a(y,z,\xi)\dslash y.
$$
Let us fix some $x\in\Rn$, with corresponding equivalence
class $[x]\in\Tn$ which we still denote by $x$.
Let $V\subset\Rn$ be an 
open cube
with side length equal to $2\pi$ centred at $x$.
Let $\chi=\chi(x,y)\in C^\infty(\Tn\times\Tn)$ be such that
$0\leq\chi\leq 1$,
$\chi(x,y)=1$ for $|x-y|<\kappa$ for some sufficiently
small $\kappa>0$, and 
such that $\supp\chi(x,\cdot)\cap \overline{V} \subset V$.
The last condition means that 
$\chi(x,\cdot)|_V\in C_0^\infty(V)$ is
supported away from the boundaries of the cube $V$.
Let
\[
\begin{aligned}
c(x,z,\xi) & = c^I(x,z,\xi)+c^{II}(x,z,\xi), \textrm{ where } \\
c^I(x,z,\xi) & = \sum_{\eta\in\Zn}\int_\Tn 
\erm^{\irm(\phi(y,\xi)-\phi(x,\xi)+(x-y)\cdot\eta)}
\ (1-\chi(x,y))\  a(y,z,\xi)\ p(x,\eta) \dslash y, \\
c^{II}(x,z,\xi) & = \sum_{\eta\in\Zn} \int_\Tn 
\erm^{\irm(\phi(y,\xi)-\phi(x,\xi)+(x-y)\cdot\eta)}\ \chi(x,y)
\ a(y,z,\xi)\ p(x,\eta) \dslash y.
\end{aligned}
\]

{1. Estimate on the support of $1-\chi$.} 
By making a decomposition into cones centred at $x$ viewed
as a point in $\Rn$, it follows that we can assume without
loss of generality that the support of $1-\chi$ is contained
in a set where $C<|x_j-y_j|$, for some
$1\leq j\leq n$. In turn, because of the assumption
on the support of $\chi(x,\cdot)|_V$ it follows that
$C<|x_j-y_j|<2\pi-C$, for some $C>0$.
Now we are going to apply the summation by parts formula
\eqref{EQ:summation-by-parts} to estimate $c^I(x,z,\xi)$. First we
notice that it follows that 
$$\triangle_{\eta_j} \erm^{\irm (x-y)\cdot\eta}=
\erm^{\irm (x-y)\cdot(\eta+e_j)}-\erm^{\irm (x-y)\cdot\eta}=
\erm^{\irm (x-y)\cdot\eta} \p{\erm^{\irm (x_j-y_j)}-1}\not=0
\;\textrm{ on } \supp (1-\chi).$$
Hence by the summation by parts formula 
\eqref{EQ:summation-by-parts} we get that
\begin{equation}\label{EQ:torus-comp-PT1}
  \sum_{\eta\in\Zn}  
 \erm^{\irm (x-y)\cdot\eta}\ p(x,\eta)=
 \p{\erm^{\irm (x_j-y_j)}-1}^{-N_1}
 \sum_{\eta\in\Zn} \erm^{\irm (x-y)\cdot\eta} 
 \ \overline{\triangle_{\eta_j}}^{N_1} p(x,\eta);
\end{equation}
here sum on the right hand side
converges absolutely for large enough $N_1$.
On the other hand, we can integrate by parts with
operator 
\[
^tL_y=\frac{1-\Lap_y}{\jp{\nabla_y\phi(y,\xi)}^2-
\irm\ \Lap_y\phi(y,\xi)},
\]
where $\Lap_y$ is the Laplace operator with 
respect to $y$, and for which we have
$L_y^{N_2}\erm^{\irm\phi(y,\xi)}=\erm^{\irm\phi(y,\xi)}.$ 
Note that in view of
our assumption \eqref{EQ:torus-compPT-phi} on $\phi$, we have
\[
|\jp{\nabla_y\phi(y,\xi)}^2-\irm\ \Lap_y\phi(y,\xi)|\geq
|\jp{\nabla_y\phi(y,\xi)}|^2\geq C_1\jp{\xi}^2.
\]
Therefore,
\begin{multline*}
c^{I}(x,z,\xi) = \sum_{\eta\in\Zn} \int_\Tn 
\erm^{\irm(\phi(y,\xi)-\phi(x,\xi)+x\cdot\eta)}
L_y^{N_2}\left\{
\erm^{-\irm y\cdot\eta}\times \right. \\
\left. \times
 \p{\erm^{\irm (x_j-y_j)}-1}^{-N_1}
 \overline{\triangle_{\eta_j}}^{N_1} p(x,\eta)\
 (1-\chi(x,y))\ a(y,z,\xi) \right\} \dslash y.
\end{multline*}
From the properties of amplitudes, we get 
\begin{equation*}
|c^{I}(x,z,\xi)|  \leq  C\sum_{\eta\in\Zn} 
\int_{V\cap\{2\pi-c>|x_j-y_j|>c\}} \jp{\xi}^{m-2N_2}
\jp{\eta}^{2N_2+\ell-N_1} \dslash y \leq 
C \jp{\xi}^{-N}
\end{equation*}
for all $N$, if we choose large enough $N_2$ and 
then large enough $N_1$.
We can easily see that similar estimates work for
the derivatives of $c^I$, completing the proof on the
support of $1-\chi$.

{2. Estimate on the support of $\chi$.} 
Extending $\widetilde{p}\in S^\ell(\Tn\times\Zn)$ to a symbol in 
$p\in S^\ell(\Tn\times\Rn)$ as given in Theorem
\ref{THM:torus-extend-symbols}, we will make its usual Taylor 
expansion at $\eta=\nabla_x\phi(x,\xi)$, so that we have 
\begin{equation}\label{EQ:ralpha}
\begin{aligned}
& p(x,\eta)  = \sum_{|\alpha|<M} 
\frac{(\eta-\nabla_x\phi(x,\xi))^\alpha}{\alpha !}
\ \partial_\xi^\alpha p(x,\nabla_x\phi(x,\xi)) + \\
& \qquad \qquad + \sum_{|\alpha|=M}
C_\alpha\ {(\eta-\nabla_x\phi(x,\xi))^\alpha} \
r_\alpha(x,\xi,\eta-\nabla_x\phi(x,\xi)), \\
& r_\alpha(x,\xi,\eta-\nabla_x\phi(x,\xi))  = \int_0^1 (1-t)^{M_1}
\ \partial_\xi^{\alpha} p(x,t\eta+(1-t)
\nabla_x\phi(x,\xi)) \drm t. 
\end{aligned}
\end{equation}
Then
\[
\begin{aligned}
c^{II}(x,z,\xi) = 
\sum_{|\alpha|<M} \frac{1}{\alpha!}
 \ c_\alpha(x,z,\xi)+\sum_{|\alpha|=M} C_\alpha R_\alpha(x,z,\xi),
\end{aligned}
\]
where
\[
\begin{aligned}
R_{\alpha}(x,z,\xi) & = \sum_{\eta\in\Zn} \int_\Tn 
{\rm e}^{{\rm i}(\phi(y,\xi)-\phi(x,\xi)+(x-y)\cdot\eta)}
(\eta-\nabla_x\phi(x,\xi))^{\alpha}\times \\
 & \qquad \qquad \times \chi(x,y)\
r_\alpha(x,\xi,\eta-\nabla_x\phi(x,\xi))\ a(y,z,\xi) \dslash y, \\
c_{\alpha}(x,z,\xi) & = \sum_{\eta\in\Zn}
\int_\Tn 
{\rm e}^{{\rm i}(\phi(y,\xi)-\phi(x,\xi)+(x-y)\cdot\eta)}
 (\eta-\nabla_x\phi(x,\xi))^{\alpha} \times \\
& \quad\quad \times\chi(x,y)\
 a(y,z,\xi)\
\partial_\xi^\alpha
p(x,\nabla_x\phi(x,\xi))  \dslash y.
\end{aligned}
\]
Now using Corollary \ref{COR:torus-period-restr} 
we can calculate
$$
\begin{aligned}
c_\alpha(x,z,\xi) & = 
   \partial_\xi^\alpha p(x,\nabla_x\phi(x,\xi))\ 
  \br{D_y-\nabla_x\phi(x,\xi)}^\alpha
  \b{{\rm e}^{{\rm i}(\phi(y,\xi)-\phi(x,\xi))}\
   \chi(x,y)\ a(y,z,\xi)}|_{y=x} \\
   & = \partial_\xi^\alpha p(x,\nabla_x\phi(x,\xi)) \times \\
   & \qquad \times
   \int_\Rn \int_V \erm^{\irm (x-y)\cdot\eta}\
   \br{\eta-\nabla_x\phi(x,\xi)}^\alpha\
   {\rm e}^{{\rm i}(\phi(y,\xi)-\phi(x,\xi))}\
   \chi(x,y)\ a(y,z,\xi) \dslash y \drm\eta \\
  & = \partial_\xi^\alpha p(x,\nabla_x\phi(x,\xi))\ D_y^{\alpha}
  \br{{\rm e}^{{\rm i}(\phi(y,\xi)-\phi(x,\xi)
   +(x-y)\cdot\nabla_x\phi(x,\xi)}\chi(x,y)\ a(y,z,\xi)}|_{y=x}.
\end{aligned}
$$
Since $\chi$ is identically equal to one for $y$ near $x$,
we obtain the asymptotic formula \eqref{EQ:torus-PT-asympt},
once the remainders $R_\alpha$ are analyzed.

3. Estimates for the remainder. 
Let us first write the remainder in the form
\begin{equation}\label{EQ:torus-PT-rmd1}
\begin{aligned}
R_{\alpha}(x,z,\xi) & = \sum_{\eta\in\Zn} \int_\Tn 
{\rm e}^{{\rm i}(x-y)\cdot\eta}\
r_\alpha(x,\xi,\eta-\nabla_x\phi(x,\xi)) \times \\
 & \qquad \qquad \times 
(\eta-\nabla_x\phi(x,\xi))^{\alpha}\ \chi(x,y)\ g(y)  
 \dslash y, 
\end{aligned}
\end{equation} 
with 
$$
 g(y)={\rm e}^{{\rm i}(\phi(y,\xi)-\phi(x,\xi))}\
 \chi(x,y)\ a(y,z,\xi),
$$
which is a $2\pi$-periodic function of $y$. 
Now, we can use Corollary \ref{COR:torus-period-restr} 
to conclude that 
$R_\alpha(x,z,\xi)$ in \eqref{EQ:torus-PT-rmd1}
is equal to the periodization with respect to $x$ in the form
$R_{\alpha}(x,z,\xi)=\Pcal_x S_{\alpha}(x,z,\xi)$, where
\begin{equation}\label{EQ:torus-PT-rmd2}
 \begin{aligned}
S_{\alpha}(x,z,\xi) & = 
r_\alpha(x,\xi,D_y-\nabla_x\phi(x,\xi))  
\ (D_y-\nabla_x\phi(x,\xi))^{\alpha} g(y)|_{y=x} \\
& = 
 \int_\Rn \int_V 
{\rm e}^{{\rm i}(x-y)\cdot\eta}
\ r_\alpha(x,\xi,\eta-\nabla_x\phi(x,\xi)) \times \\
 & \qquad \qquad \times 
(\eta-\nabla_x\phi(x,\xi))^{\alpha}\ \chi(x,y)\ g(y)  
 \dslash y \drm\eta \\
& = \int_\Rn\int_V \erm^{\irm(x-y)\cdot\theta}
\erm^{\irm\Psi(x,y,\xi)}\ \theta^{\alpha}\ \chi(x,y)\
a(y,z,\xi)\ r_\alpha(x,\xi,\theta) \dslash y 
\drm\theta,
\end{aligned}
\end{equation} 
where we changed the variables by 
$\theta=\eta-\nabla_x\phi(x,\xi)$ and where 
$$
 \Psi(x,y,\xi) := \phi(y,\xi) - \phi(x,\xi) + 
  (x-y)\cdot\nabla_x\phi(x,\xi)
$$
and 
$$
r_\alpha(x,\xi,\theta)  = \int_0^1 (1-t)^{M_1}\
\partial_\xi^{\alpha} p(x,\nabla_x\phi(x,\xi)+t\theta) \drm t. 
$$
Since the periodization
$\Pcal_x$ does not change the orders in $z$ and $\xi$ it is
enough to derive the required estimates for 
$S_{\alpha}(x,z,\xi)$.

Let $\rho\in
C_0^\infty(\Rn)$ be such that $\rho(\theta)=1$ for 
$|\theta|<\epsilon/2$, and
$\rho(\theta)=0$ for $|\theta|>\epsilon$,
for some small $\epsilon>0$ to be chosen later. We decompose
\begin{equation}\label{EQ:Ralphas}
\begin{aligned}
S_{\alpha}(x,z,\xi) & = S_\alpha^I(x,z,\xi)+
 S_\alpha^{II}(x,z,\xi), \textrm{ where } \\
S_\alpha^I(x,z,\xi) & = \int_\Rn\int_v \erm^{\irm(x-y)\cdot\theta}
\ \rho\left(\frac{\theta}{\jp{\xi}}\right)\times
\\ & \qquad\qquad \times r_\alpha(x,\xi,\theta)\ D_y^{\alpha} \left[
\erm^{\irm\Psi(x,y,\xi)}\
\chi(x,y)\ a(y,z,\xi) \right] \dslash y 
\drm\theta \\
S_\alpha^{II}(x,z,\xi) & 
= \int_\Rn\int_V \erm^{\irm(x-y)\cdot\theta} \left(1-\rho\left(
\frac{\theta}{\jp{\xi}}\right)\right)\times \\
& \qquad \qquad \times r_\alpha(x,\xi,\theta)\
D_y^{\alpha}\left[ 
\erm^{\irm\Psi(x,y,\xi)}\ \chi(x,y)\ a(y,z,\xi)
\right] \dslash y \drm\theta.
\end{aligned}
\end{equation}

{3.1. Estimate for $|\theta|\leq \epsilon\jp{\xi}.$} For
sufficiently small $\epsilon>0$, for any $0\leq t\leq 1$,
$\jp{\nabla_x\phi(x,\xi)+t\theta}$ and $\jp{\xi}$ are
equivalent. Indeed, if we use the inequalities
\[
\jp{z}\leq 1+|z|\leq \sqrt{2}\jp{z},
\]
we get
\[
\begin{aligned}
\jp{\nabla_x\phi(x,\xi)+t\theta}\leq & (C_2\sqrt{2}+\epsilon)\jp{\xi} \\
\sqrt{2}\jp{\nabla_x\phi(x,\xi)+t\theta}\geq & 1+
|\nabla_x\phi|-|\theta|
\geq \jp{\nabla_x\phi}-|\theta|\geq (C_1-\epsilon)\jp{\xi},
\end{aligned}
\]
so we will take $\epsilon<C_1.$ This equivalence means that for
$|\theta|\leq\epsilon\jp{\xi}$, function
$r_\alpha(x,\xi,\theta)$ is dominated by 
$\jp{\xi}^{\ell-|\alpha|}$ since $p\in S^\ell(\Tn\times\Rn).$
We will need two auxiliary estimates. The first
estimate

\begin{equation}\label{eq:estr}
\begin{aligned}
\left|\partial_\theta^{\gamma}\left( \rho\left(
\frac{\theta}{\jp{\xi}}\right)
r_\alpha(x,\xi,\theta)\right)\right| & \leq
C\sum_{\delta\leq\gamma} \left|\partial_\theta^\delta
\rho\left(\frac{\theta}{\jp{\xi}}\right)
\partial_\theta^{\gamma-\delta}
r_\alpha(x,\xi,\theta)\right| \\
& \leq C\sum_{\delta\leq\gamma}\jp{\xi}^{-|\delta|}
\jp{\xi}^{\ell-|\alpha|-|\gamma-\delta|} \\
& \leq C\jp{\xi}^{\ell-|\alpha|-|\gamma|}
\end{aligned}
\end{equation}
follows from the properties of $r_\alpha$. Before we state the
second estimate, let us analyze the structure of
$\partial_y^\alpha \erm^{\irm\Psi(x,y,\xi)}.$ It has at most $|\alpha|$
powers of terms $\nabla_y\phi(y,\xi)-\nabla_x\phi(x,\xi)$,
possibly also multiplied by at most $|\alpha|$ higher order
derivatives $\partial_y^\delta\phi(y,\xi)$.
The product of the terms of the 
form $\nabla_y\phi(y,\xi)-\nabla_x\phi(x,\xi)$
can be estimated
by $C(|y-x|\jp{\xi})^{|\alpha|}$. The terms
containing no difference $\nabla_y\phi(y,\xi)-\nabla_x\phi(x,\xi)$
are the products of at most $|\alpha|/2$ terms of the type
$\partial_y^\delta\phi(y,\xi)$, and the product
of all such terms can be estimated by
$C\jp{\xi}^{|\alpha|/2}$. Altogether,
we obtain the estimate
\[
|\partial_y^\alpha \erm^{\irm\Psi(x,y,\xi)}|\leq C_\alpha
\ (1+\jp{\xi}|y-x|)^{|\alpha|}\ \jp{\xi}^{|\alpha|/2}.
\]
The second auxiliary estimate now is
\begin{equation}\label{eq:ests}
\left|D_y^{\alpha}\left[ \erm^{\irm\Psi(x,y,\xi)}\ \chi(x,y)
\ a(y,z,\xi)\right] \right|
\\ \leq C_\alpha\
(1+\jp{\xi}|y-x|)^{|\alpha|}\ \jp{\xi}^{\frac{|\alpha|}{2}+m}.
\end{equation}
Now we are ready to prove a necessary estimate for
$S^I_{\alpha}(x,z,\xi).$ Let
\[
L_\theta=\frac{(1-\jp{\xi}^2\Lap_\theta)}{1+\jp{\xi}^2|x-y|^2},\;\;
L_\theta^N \erm^{\irm(x-y)\cdot\theta}=\erm^{\irm(x-y)\cdot\theta},
\]
where $\Lap_\theta$ is the Laplace operator with respect to
$\theta$.
Integrations by parts with $L_\theta$ yield
\[
\begin{aligned}
S^I_{\alpha}(x,z,\xi) = & \int_\Rn\int_V
\frac{\erm^{\irm(x-y)\cdot\theta}}{(1+\jp{\xi}^2 |x-y|^2)^N}
(1-\jp{\xi}^{2}\Lap_\theta)^N \\
& \;\; \left\{ \rho\left(\frac{\theta}{\jp{\xi}}\right)
r_\alpha(x,\xi,\theta)\ D_y^{\alpha}\left[ 
\erm^{\irm\Psi(x,y,\xi)}\ \chi(x,y)\ a(y,z,\xi)
\right]
\right\} \dslash y \drm\theta \\
= & \int_\Rn \int_V 
\frac{\erm^{\irm(x-y)\cdot\theta}}{(1+\jp{\xi}^2 |x-y|^2)^N}
\sum_{{|\gamma|\leq 2N}}
 C_{\gamma} \jp{\xi}^{|\gamma|} \\
& \;\;
\left\{
D_y^{\alpha}\left[ \erm^{\irm\Psi(x,y,\xi)}
\ \chi(x,y)\ a(y,z,\xi)\right]
\partial_\theta^{\gamma}\left(
\rho\left(\frac{\theta}{\jp{\xi}}\right)
r_\alpha(x,\xi,\theta)\right)
\right\} \dslash y \drm\theta.
\end{aligned}
\]
Using estimates (\ref{eq:estr}), (\ref{eq:ests})
and the fact that the
measure of the support of 
function $\theta\mapsto\rho(\theta/\jp{\xi})$ is
estimated by $(\epsilon\jp{\xi})^n$, we obtain the estimate
\[
\begin{aligned}
|S^I_{\alpha}(x,z,\xi)| & \leq
C\sum_{|\gamma|\leq 2N}
\jp{\xi}^{n+|\gamma|+\frac{|\alpha|}{2}+m}
\jp{\xi}^{\ell-|\alpha|-|\gamma|} 
\int_{V} \frac{(1+\jp{\xi}|y-x|)^{|\alpha|}}
{(1+\jp{\xi}^2 |x-y|^2)^N} \dslash y \\
& \leq C\jp{\xi}^{\ell+m+n-\frac{|\alpha|}{2}},
\end{aligned}
\]
if we choose $N$ large enough, e.g. $N\geq M=|\alpha|$. 

Each derivative of $S_{\alpha}^I(x,z,\xi)$ with respect to $x$
or $\xi$ gives an extra power of $\theta$ under the integral. 
Integrating by parts, this
amounts to taking more $y$-derivatives, giving a higher power of
$\jp{\xi}.$ However, this is not a problem if for the estimate
for a given number of derivatives of the remainder
$S_{\alpha}^I(x,z,\xi)$, we choose $M=|\alpha|$
sufficiently large.

{3.2. Estimate for $|\theta|>\epsilon\jp{\xi}.$}
Let us define
\[
\omega(x,y,\xi,\theta)=(x-y)\cdot\theta+\Psi(x,y,\xi)=
(x-y)\cdot(\nabla_x\phi(x,\xi)+\theta)+\phi(y,\xi)-\phi(x,\xi).
\]
From \eqref{EQ:torus-compPT-phi} and
\eqref{EQ:torus-compPT-phi2} we have
\begin{equation}\label{eq:rho}
\begin{aligned}
|\nabla_y\omega| & =|-\theta+\nabla_y\phi-\nabla_x\phi|\leq
2C_2(|\theta|+\jp{\xi}), \\
|\nabla_y\omega| & \geq |\theta|-|\nabla_y\phi-\nabla_x\phi| \geq
\frac{1}{2}|\theta|+\left(\frac{\epsilon}{2}-C_0|x-y|
\right)\jp{\xi}\geq
C(|\theta|+\jp{\xi}),
\end{aligned}
\end{equation}
if we choose $\kappa<\frac{\epsilon}{2C_0},$ 
since $|x-y|<\kappa$ in the support of $\chi$ in $V$
(recall that we were free to choose $\kappa>0$).
Let us denote
\[
\sigma_{\gamma_1}(x,y,\xi):= \erm^{-\irm\Psi(x,y,\xi)}\ D_y^{\gamma_1}
\erm^{\irm\Psi(x,y,\xi)}.
\]
For any $\nu$ we have an estimate
\begin{equation}\label{EQ:torus-PT-def-sig}
|\partial_y^\nu\sigma_{\gamma_1}(x,y,\xi)|
\leq C\jp{\xi}^{|\gamma_1|},
\end{equation} 
because of our assumption \eqref{EQ:torus-compPT-phi2} that
$|\partial_y^\nu\phi(y,\xi)|\leq C_\nu\jp{\xi}.$ For
$M=|\alpha|>\ell$  we also observe that
\begin{equation}\label{eq:rs}
|r_\alpha(x,\xi,\theta)|\leq C_\alpha, \;
|\partial_y^\nu a(y,z,\xi)|\leq
C_\beta\jp{\xi}^{m}.
\end{equation}
Let us take $^tL_y=\irm|\nabla_y\omega|^{-2}\sum_{j=1}^n
(\partial_{y_j}\omega)
\partial_{y_j}.$ It can be shown by induction that operator
$L_y^N$ has the form
\begin{equation}\label{EQ:LNy}
L_y^N=\frac{1}{|\nabla_y\omega|^{4N}}\sum_{|\nu|\leq N}
P_{\nu,N}\partial_y^\nu,\;\; P_{\nu,N}=\sum_{|\mu|=2N}
c_{\nu\mu\delta_j}(\nabla_y\omega)^\mu
\partial_y^{\delta_1}\omega\cdots
\partial_y^{\delta_N}\omega,
\end{equation}
where $|\mu|=2N, |\delta_j|\geq 1, \sum_{j=1}^N |\delta_j|+|\nu|=2N.$
It follows from \eqref{EQ:torus-compPT-phi2} 
and \eqref{eq:rho} that $|P_{\nu,N}|\leq
C(|\theta|+\jp{\xi})^{3N}$, since for all $\delta_j$ we have
$|\partial_y^{\delta_j}\omega|\leq C(|\theta|+\jp{\xi}).$
By the Leibnitz formula we have
\begin{equation}\label{EQ:RII}
\begin{aligned}
& S^{II}_{\alpha}(x,z,\xi) = \\
& = \int_\Rn\int_V \erm^{\irm(x-y)\cdot\theta} \left(1-\rho\left(
\frac{\theta}{\jp{\xi}}\right)\right)
r_\alpha(x,\xi,\theta) \times \\
& \;\;\; \times D_y^{\alpha}\left[
\erm^{\irm\Psi(x,y,\xi)}\ \chi(x,y)\
a(y,z,\xi)\right] \dslash y \drm\theta \\
& = \int_\Rn\int_V \erm^{\irm\omega(x,y,\xi,\theta)}
\left(1-\rho\left(\frac{\theta}{\jp{\xi}}\right)\right) 
r_\alpha(x,\xi,\theta) \times
\\
& \; \; \; \times \sum_{\gamma_1+\gamma_2+\gamma_3=\alpha}
\sigma_{\gamma_1}(x,y,\xi)\ D_y^{\gamma_2}\ \chi(x,y)\
D_y^{\gamma_3}a(y,z,\xi) \dslash y \drm\theta \\
& = \int_\Rn\int_V \erm^{\irm\omega(x,y,\xi,\theta)}
|\nabla_y\omega|^{-4N}\sum_{|\nu|\leq N} P_{\nu,N}(x,y,\xi,\theta)
\left(1-\rho\left(\frac{\theta}{\jp{\xi}}
\right)\right)\times \\
& \;\;\; \times  r_\alpha(x,\xi,\theta)
\sum_{\gamma_1+\gamma_2+\gamma_3=\alpha}
\partial_y^\nu\left(\sigma_{\gamma_1}(x,y,\xi)\
D_y^{\gamma_2}\ \chi(x,y)\
D_y^{\gamma_3}a(y,z,\xi)\right) \dslash y \drm\theta.
\end{aligned}
\end{equation}
It follows now from \eqref{EQ:torus-PT-def-sig} 
and (\ref{eq:rs}) that
\[
\begin{aligned}
|S^{II}_{\alpha}(x,z,\xi)| & \leq C
\int_{|\theta|>\epsilon\jp{\xi}/2} 
(|\theta|+\jp{\xi})^{-N} 
\jp{\xi}^{|\alpha|}
\jp{\xi}^{m} \drm\theta \\
& \leq C \jp{\xi}^{m+|\alpha|+n-N},
\end{aligned}
\]
which yields the desired estimate if we take large enough $N$.
For the derivatives of $S_{\alpha}^{II}(x,z,\xi)$, similar to
Part 3.1 for $S_{\alpha}^I$, we can get extra powers of
$\theta$, which can be taken care of by choosing large $N$. 
The proof of Theorem \ref{THM:torus-PT} is now complete.
\end{proof}

\begin{rem}\label{REM:torus-PT}
We could also use the following asymptotic expansion
for $c$ based on the discrete Taylor expansion from
Theorem \ref{THM:torus-Taylor-thm}:
\begin{eqnarray*}
  c(x,z,\xi)
  & \sim &
  \sum_{\theta\in\Zn} \sum_{\alpha\geq 0}
  \frac{1}{\alpha!}\ \theta^{(\alpha)}
  \left[ \triangle_\omega^\alpha p(x,\omega)
  \right]_{\omega=\nabla_x\phi(x,\xi)}
  \int_{\Tn} {\rm e}^{{\rm i}(\Psi(x,y,\xi)+(x-y)\cdot\theta)}
  \ a(y,z,\xi)\ \dslash y \\
  & = &
  \sum_{\alpha\geq 0}
  \frac{1}{\alpha!}
  \left[ \triangle_\omega^\alpha p(x,\omega)
  \right]_{\omega=\nabla_x\phi(x,\xi)}
  \sum_{\theta\in\Zn} 
  \int_{\Tn} \theta^{(\alpha)}
  \ {\rm e}^{{\rm i}(x-y)\cdot\theta}
  \ {\rm e}^{{\rm i}\Psi(x,y,\xi)}
  \ a(y,z,\xi)\ \dslash y \\
  & = &
  \sum_{\alpha\geq 0}
  \frac{1}{\alpha!}
  \left[ \triangle_\omega^\alpha p(x,\omega)
  \right]_{\omega=\nabla_x\phi(x,\xi)}
  D_y^{(\alpha)}
  \left[ {\rm e}^{{\rm i}\Psi(x,y,\xi)}\ 
  a(y,z,\xi) \right]_{y=x}. 
\end{eqnarray*}
\end{rem}

\section{Conditions for $L^{2}$-boundedness}
\label{SEC:torus-L2}

In this section we will discuss what conditions on 
the toroidal symbol $a$
guarantee the $L^{2}$-boundedness of the corresponding 
pseudo-differential operator
$a(X,D):C^\infty(\Tn)\to \Dcal'(\Tn)$. In Theorem 
\ref{THM:torus-L2-FSOs} we will also prove a result
on the $L^2$-boundedness of Fourier series operators,
and in Remark \ref{REM:torus-Sobolev} we will give its
consequence for the boundedness in Sobolev spaces.

\begin{theorem}\label{TH:L2-pseudos}
Let $a:\Tn\times\Zn\to\C$ be such that 
\begin{equation}\label{EQ:torus-L2-pse}
  \left| \partial_x^\beta a(x,\xi) \right| \leq
  C \quad\textrm{ for all } (x,\xi)\in\Tn\times\Zn,
\end{equation} 
and all $|\beta| \leq n+1$.
Then operator $a(X,D)$ extends to a bounded
operator on $L^{2}(\Tn)$.
\end{theorem}
We note that compared with several well-known theorems
on the $L^2$--boundedness of pseudo-differential operators
(see e.g. Calderon and Vaillancourt \cite{CV71}, 
Coifman and Meyer \cite{CM78}, Cordes \cite{Co75}), 
Theorem \ref{TH:L2-pseudos} 
does not require any regularity with respect to the
$\xi$-variable. In fact, the boundedness of all partial
differences of all orders $\geq 1$
with respect to $\xi$ follows automatically from
\eqref{EQ:torus-L2-pse}.

\begin{proof}
We can write
\begin{eqnarray*}
  a(X,D)f(x)
  & = & \sum_{\xi\in\Zn} a(x,\xi)\ \p{\FTT f}(\xi)
  \ {\rm e}^{{\rm i}x\cdot\xi} \\
  & = & \sum_{\xi,\eta\in\Zn} \p{\FTT{a}}(\eta,\xi)
  \ \p{\FTT f}(\xi)
  \ {\rm e}^{{\rm i}x\cdot(\xi+\eta)}\\
  & = & \sum_{\omega\in\Zn} {\rm e}^{{\rm i}x\cdot\omega}
  \sum_{\xi\in\Zn} \p{\FTT a}(\omega-\xi,\xi)
  \ \p{\FTT f}(\xi),
\end{eqnarray*}
where $\p{\FTT{a}}(\eta,\xi)$ is the $\eta^{th}$
Fourier coefficient 
of $a(\cdot,\xi)$.
Here
$\left|\p{\FTT a}(\eta,\xi)\right| \leq C\ 
\langle\eta\rangle^{-n-1}$,
so that
\begin{eqnarray*}
  \| a(X,D) f \|_{L^{2}(\Tn)}^2
  & = & \int_{\Tn} |a(X,D) f(x)|^2\ \dslash x \\
  & = & \sum_{\omega\in\Zn} |\FTT (a(X,D)f)(\omega)|^2 \\
  & = & \sum_{\omega\in\Zn} \left|
    \sum_{\xi\in\Zn} \p{\FTT a}(\omega-\xi,\xi)
    \ \p{\FTT f}(\xi) \right|^2 \\
  &\leq &
  \left( \sup_{\omega\in\Zn} \sum_{\xi\in\Zn}
    \left|\p{\FTT a}(\omega-\xi,\xi)\right|
    \right) \times \\
  & & \quad \times
    \left( \sup_{\xi\in\Zn} \sum_{\omega\in\Zn}
    \left|\p{\FTT a}(\omega-\xi,\xi)\right|
    \right)
    \sum_{\xi\in\Zn} \left|\p{\FTT f}(\xi) \right|^2 \\
  & \leq & C\ \|f\|_{L^{2}(\Tn)}^2
\end{eqnarray*}
by the discrete version of Young's inequality.
\end{proof}

As an extension of this result to the setting of
Fourier series operators, we have the following
$L^2$--boundedness property.

\begin{theorem}\label{THM:torus-L2-FSOs}
Let $T:C^\infty(\Tn)\to \Dcal'(\Tn)$ be defined
by $$Tu(x)=\sum_{k\in\Zn}\ \erm^{\irm\phi(x,k)}\ a(x,k)\ 
\p{\FTT{u}}(k),$$
where $\phi:\Rn\times\Zn\to\R$ and
$a:\Tn\times\Zn\to\C$.
Assume that function
$x\mapsto{\rm e}^{\irm\phi(x,\xi)}$ is $2\pi$-periodic for
every $\xi\in\Zn$, and that 
for all $|\alpha|\leq 2n+1$ and $|\beta|=1$ we have
\begin{equation}\label{l2a1}
  \left| \partial_x^\alpha a(x,k) \right| \leq
  C \textrm{ and }
  \left| \partial_x^\alpha\triangle_k^\beta \phi(x,k) \right| \leq
  C\;\; \textrm{ for all } (x,k)\in \Tn\times\Zn.
\end{equation}
Assume also that 
\begin{equation}\label{l2a2}
  \left| \nabla_x\phi(x,k)-\nabla_x\phi(x,l) \right| \geq
  C|k-l|\; \textrm{ for all } x\in\Tn, \; k,l\in\Zn.
\end{equation}
Then $T$ extends to a bounded operator 
on $L^{2}(\Tn)$.
\end{theorem}

Note that  condition \eqref{l2a2} is a discrete version of
the usual local graph condition for Fourier integral operators,
necessary for the local $L^2$-boundedness. We also note that
conditions on the boundedness of higher order differences
of phase and amplitude would follow automatically from condition
\eqref{l2a1}. Therefore, this theorem relaxes assumptions
on the behaviour with respect to the dual variable,
compared, for example, with the corresponding result
by Ruzhansky and Sugimoto \cite{RS06a}.

\begin{proof} 
Since for $u:\Tn\to\C$ we have
$\|u\|_{L^2(\Tn)}=\|\FTT{u}\|_{\ell^2(\Zn)},$
it is enough to prove that operator
$$Sw(x)=\sum_{k\in\Zn} \erm^{\irm\phi(x,k)}\ a(x,k)\ w(k)$$
is bounded from $\ell^2(\Zn)$ to $L^2(\Tn)$. 
Let us define $S_l w(x)=\erm^{\irm\phi(x,l)}\ a(x,l)\ w(l),$
so that $S=\sum_{l\in\Zn} S_l$.
From the identity
$$
\p{w,S^*v}_{\ell^2(\Zn)}=\p{Sw,v}_{L^2(\Tn)}=
\int_\Tn \sum_{k\in\Zn} \erm^{\irm\phi(x,k)}\ a(x,k)\ w(k)\ 
\overline{v(x)} \dslash x
$$
we find that the adjoint $S^*$ to $S$ is given by
$$(S^*v)(k)=\int_{\Tn} \erm^{-\irm\phi(x,k)}\ 
\overline{a(x,k)}\ v(x) 
\dslash x$$
and so we also have
$$(S_l^*v)(m)=\delta_{lm}
\int_{\Tn} \erm^{-\irm\phi(x,m)}\ \overline{a(x,m)}\
v(x) \dslash x=
\delta_{lm} (S^*v)(l).$$
It follows that
\begin{eqnarray*}
  S_k S_l^* v(x)
  & = & \erm^{\irm\phi(x,k)}\ a(x,k)\ (S_l^* v)(k) \\
  & = & \delta_{lk} \int_\Tn \erm^{\irm\phi(x,k)}\ a(x,k)\
    \erm^{-\irm\phi(y,k)}\ \overline{a(y,k)}\ v(y) \dslash y \\
  & = & \int_\Tn K_{kl}(x,y)\ v(y) \dslash y,
\end{eqnarray*}
where $K_{kl}(x,y)=\delta_{kl} \erm^{\irm[\phi(x,k)-\phi(y,k)]}
\ a(x,k)\ \overline{a(y,k)}.$ From \eqref{l2a1} and compactness
of the torus it follows that
the kernel $K_{kl}$ is bounded and that
$\|S_k S_l^* v\|_{L^2(\Tn)}\leq C\delta_{kl}\|v\|_{L^2(\Tn)}.$
In particular, we can trivially conclude that for any $N\geq 0$ we
have the estimate
\begin{equation}\label{EQ:FSO-L2-Ikl}
 \n{S_k S_l^*}_{L^2(\Tn)\to L^2(\Tn)} \leq \frac{C_N}
 {1+|k-l|^N}.
\end{equation} 
On the other hand, we have
\begin{eqnarray*}
  (S_l^* S_k w)(m)
  & = & \delta_{lm} \int_{\Tn} \erm^{-\irm\phi(x,l)}\ 
     \overline{a(x,l)}\ (S_k w)(x) \dslash x \\
  & = & \delta_{lm} \int_{\Tn} \erm^{\irm[\phi(x,k)-\phi(x,l)]} 
    \ a(x,k)\ \overline{a(x,l)}\ w(k) \dslash x \\
  & = &
   \sum_{\mu\in\Zn} \widetilde{K_{lk}}(m,\mu)\ w(\mu),
\end{eqnarray*}
where $\widetilde{K_{lk}}(m,\mu)=
\delta_{lm}\delta_{k\mu}\int_{\Tn} \erm^{\irm[\phi(x,k)-\phi(x,l)]}
 \ a(x,k)\ \overline{a(x,l)} \dslash x.$
Now, if $k\not=l$, integrating by parts $(2n+1)$--times with operator
$\frac{1}{\irm}\frac{\nabla_x\phi(x,k)-\nabla_x\phi(x,l)}
{|\nabla_x\phi(x,k)-\nabla_x\phi(x,l)|^2}\cdot\nabla_x$
and using the periodicity of $a$ and $\nabla_x\phi$
(so there are no boundary terms), we get the estimate
\begin{equation}\label{EQ:FSO-L2-ker1}
|\widetilde{K_{lk}}(m,\mu)|\leq 
{\frac{C\ \delta_{lm}\ \delta_{k\mu}}{1+|k-l|^{2n+1}}},
\end{equation} 
where we also used that by the discrete Taylor theorem 
\ref{THM:torus-Taylor-thm} the
second condition in \eqref{l2a1} implies that
$$
  \left| \nabla_x\phi(x,k)-\nabla_x\phi(x,l) \right| \leq
  C|k-l|\; \textrm{ for all } x\in\Tn, \; k,l\in\Zn.
$$
Estimate \eqref{EQ:FSO-L2-ker1} implies 
$$
\sup_m \sum_\mu |\widetilde{K_{lk}}(m,\mu)|=
|\widetilde{K_{lk}}(l,k)| \leq 
{\frac{C}{1+|k-l|^{2n+1}}},
$$
and similarly for $\sup_\mu \sum_m$,
so that we have
\begin{equation}\label{EQ:FSO-L2-Ilk}
 \n{S_l^* S_k}_{\ell^2(\Zn)\to\ell^2(\Zn)} \leq \frac{C}
 {1+|k-l|^{2n+1}}.
\end{equation} 
These estimates for norms
$\n{S_k S_l^*}_{L^2(\Tn)\to L^2(\Tn)}$ and
$\n{S_l^* S_k}_{\ell^2(\Zn)\to\ell^2(\Zn)}$ in
\eqref{EQ:FSO-L2-Ikl} and \eqref{EQ:FSO-L2-Ilk},
respectively,
imply the theorem by a modification of Cotlar's lemma given in
Proposition \ref{PROP:torus-Cotlar}, which we use
with $\Hcal=\ell^2(\Zn)$ and $\Gcal=L^2(\Tn)$. 
\end{proof}

The following statement is a modification of the well-known
Cotlar's lemma taking into account the fact the operators
in our application act on functions on different Hilbert spaces.
The proof follows \cite[p. 280]{St93book} but there
is a difference in how we estimate operator norms because
we cannot immediately
replace operator $S$ by $S^* S$ in the estimates
since they act on functions on different spaces. We omit
the proof.

\begin{prop}[Variant of Cotlar's lemma]\label{PROP:torus-Cotlar}
Let $\Hcal,\Gcal$ be Hilbert spaces.
Assume that a family of bounded linear operators 
$\{S_j:\Hcal\to\Gcal\}_{j\in\Z^r}$
and positive constants $\{\gamma(j)\}_{j\in\Z^r}$ satisfy
\[
\n{S_l^*S_k}_{\Hcal\to\Hcal}\leq \left[\gamma(l-k)\right]^2,\qquad
\n{S_l S_k^*}_{\Gcal\to\Gcal}\leq \left[\gamma(l-k)\right]^2,\qquad
\]
and
$
A=\sum_{j\in\Z^r}\gamma(j)<\infty.
$
Then the operator
$
S=\sum_{j\in\Z^r}S_j
$
satisfies
$
\n{S}_{\Hcal\to\Gcal}\leq A.
$
\end{prop}

\begin{rem}[Sobolev spaces]\label{REM:torus-Sobolev}
 By using pseudo-differential operator 
 $\Lap_s(D)$ with toroidal
 symbol $(1+|\xi|^2)^{s/2}\in S^s_{1,0}(\Tn\times\Zn)$,
 $s\in\R$,
 we can define Sobolev spaces $H^s(\Tn)$ on the torus as spaces
 of all $f\in \Dcal'(\Tn)$ such that
 $\Lap_s(D)f\in L^2(\Tn)$. Then using Theorems
 \ref{THM:torus-TP}, \ref{THM:torus-PT}, and
 \ref{THM:torus-L2-FSOs}, we obtain the
 result on the boundedness of Fourier series operators
 on Sobolev spaces. Namely, let $T$ be a Fourier series
 operator as in Theorem \ref{THM:torus-PT}. Then $T$ is
 bounded on $H^s(\Tn)$ for all $s\in\R$. 
\end{rem}

\section{Applications to hyperbolic equations}
\label{SEC:torus-hyperbolic}

In this section we will briefly discuss how the analysis can
be applied to construct global parametrices for the hyperbolic
equations on the torus and how to embed certain problems in 
$\Rn$ into the torus. As already mentioned in the introduction,
the finite propagation speed of singularities in hyperbolic
equations allows one to cut-off the equation and the Cauchy data
for large $x$ for the local analysis of singularities of 
solutions for bounded times. Then the problem
can be embedded into $\Tn$ in order to apply the periodic
analysis developed in this paper. One of the advantages
of this procedure is that since phases and amplitudes now
are only evaluated at $\xi\in\Zn$ one can apply this also for
problems with low regularity in $\xi$, in particular
to problems for weakly hyperbolic
equations or systems with variable multiplicities.  
For example, if the
principal part has constant coefficients then the loss of
regularity occurs only in $\xi$ so techniques developed in
this paper can be applied. 

Let $a(X,D)$ be a pseudo-differential operator
with symbol $a=a(x,\xi)\in S^m(\Rn\times\Rn)$
(with some properties to be specified).
There is no difference in the subsequent
argument if $a=a(t,x,\xi)$ also depends on $t$.
For a function $u=u(t,x)$ of $t\in\R$ and $x\in\Rn$ we write
\begin{eqnarray*}
  a(X,D) u(t,x) & = & \int_{\Rn} a(x,\xi)
  \ \p{\FTR u}(t,\xi)\ {\rm e}^{{\rm i}x\cdot\xi}\ {\rm d}\xi \\
  & = & \int_{\Rn} \int_{\Rn} {\rm e}^{{\rm i}(x-y)\cdot\xi}
  \ a(x,\xi)\ u(t,y) \dslash y\ {\rm d}\xi.
\end{eqnarray*}
Let $u(t,\cdot)\in L^{1}(\Rn)$ ($0<t<t_0$) be a solution
to the hyperbolic problem
\be \label{he1}
  \begin{cases}
  {\rm i} \frac{\partial}{\partial t}u(t,x) = a(X,D) u(t,x), \\
  u(0,x) = f(x),
  \end{cases}
\ee
where $f\in L^{1}(\Rn)$ is compactly supported.

Assume now that $a(X,D)=a_1(X,D)+a_0(X,D)$ where $a_1(x,\xi)$
is $2\pi$-periodic and $a_0(x,\xi)$ is compactly supported in $x$
(assume even that $\supp a_0(\cdot,\xi)\subset [-\pi,\pi]^n$).
A simple example is a 
constant coefficient symbol $a_1(x,\xi)=a_1(\xi)$.
Let us also assume that $\supp f \subset [-\pi,\pi]^n$.

We will now describe a way to periodize problem \eqref{he1}.
According to Proposition \ref{p:p2}, we can replace \eqref{he1}
by 
\be \label{he2}
  \begin{cases}
  {\rm i} \frac{\partial}{\partial t}u(t,x) = 
  (a_1(x,D)+(\Pcal a_0)(X,D)) u(t,x)+Ru(t,x), \\
  u(x,0) = f(x),
  \end{cases}
\ee
where the symbol $a_1+\Pcal a_0$ is periodic and $R$ is a smoothing
operator. To study singularities of \eqref{he1},
it is sufficient to analyze the Cauchy problem
\be \label{he3}
  \begin{cases}
  {\rm i} \frac{\partial}{\partial t}v(t,x) = 
  (a_1(x,D)+(\Pcal a_0)(X,D)) v(t,x), \\
  v(x,0) = f(x)
  \end{cases}
\ee
since by Duhamel's formula we have ${\rm WF}(u-v)=\emptyset$.
This problem can be transferred to the torus. 
Let $w(t,x)=\Pcal v(\cdot,t)(x)$. 
By Theorem \ref{THM:torus-per-symb}
it solves the Cauchy problem
\be \label{he4}
  \begin{cases}
  {\rm i} \frac{\partial}{\partial t}w(t,x) = 
  (\widetilde{a_1}(x,D)+\widetilde{\Pcal a_0}(X,D)) w(t,x), \\
  w(x,0) = \Pcal f(x).
  \end{cases}
\ee
Now, if $a\in S^1$ is of the first order,
the calculus constructed in previous sections 
yields the solution in the form
$$
  w(t,x)\equiv T_t f(x)
  =\sum_{k\in\Zn} {\rm e}^{{\rm i}\phi(t,x,k)}\ b(t,x,k)
  \ \FTT\p{\Pcal f}(k),
$$
where $\phi(t,x,\xi)$ and $b(t,x,\xi)$ satisfy discrete 
analogues of the eikonal and transport equations.
Here we note that $\FTT\p{\Pcal f}(k)=\p{\FTR f}(k)$.
We also note that in principal the phase 
$\phi(t,x,k)$ is defined for discrete values of
$k\in\Zn$, so there is no issue of regularity, making
this representation potentially applicable to low
regularity problems and weakly hyperbolic equations.

Let us give a short example. If 
the symbol $a_1(x,\xi)=a_1(\xi)$ has constant coefficients
and belongs to $S^1$,
and $a_0$ belongs to $S^0$, we
can find that the phase is given by
$\phi(t,x,k)=x\cdot k+t a_1(k).$ In particular,
$\nabla_x\phi(x,k)=k$.
Applying $a(X,D)$ to $w(t,x)=T_t f(x)$ and using the
composition formula from Theorem \ref{THM:torus-PT}
we obtain
$$
  a(X,D)T_t f(x) = \sum_{k\in\Zn} \int_{\Rn}
  {\rm e}^{{\rm i}((x-z)\cdot k +ta_1(k))}
  \ c(t,x,k)\ f(z)
  \dslash z,
$$
where
\begin{equation}\label{EQ:torus-hyperb-as}
 c(t,x,k) \sim \sum_{\alpha\geq 0}
  \frac{{\rm i}^{-|\alpha|}}{\alpha!}
    \ \left.\partial_\xi^\alpha a(x,\xi)\right|_{\xi=k}
    \ \partial_x^\alpha b(t,x,k),
\end{equation} 
since function $\Psi$ in Theorem \ref{THM:torus-PT}
vanishes. From this we can find amplitude $b$ from the
discrete version of the transport equations, details of
which we omit here.
Finally, we note that we can also have an asymptotic expansion
for the amplitude $b$ in \eqref{EQ:torus-hyperb-as}
in terms of discrete differences
$\triangle_\xi^\alpha$ and corresponding derivatives
$\partial_x^{(\alpha)}$ instead of
derivatives $\partial_\xi^\alpha$ and $\partial_x^\alpha$,
respectively, if we use Remark \ref{REM:torus-PT} instead
of Theorem \ref{THM:torus-PT}.

\end{document}